\documentclass[11pt]{article}

\usepackage{amsmath, amssymb, amsthm}
\usepackage{todonotes,enumerate}
\usepackage[margin=3cm]{geometry}
\usepackage{authblk}
\usepackage{hyperref}
\usepackage{graphicx}
\usepackage{subfig}

\hypersetup{
    colorlinks=false,
    pdfborder={0 0 0},
}

\newtheorem{theorem}{Theorem}[section]

\newtheorem{lemma}[theorem]{Lemma}
\newtheorem{proposition}[theorem]{Proposition}

\newtheorem{definition}[theorem]{Definition}
\theoremstyle{definition}
\newtheorem{remark}[theorem]{Remark}

\title{Relative Geodesics in the Special Euclidean Group}
\author[1]{Darryl D. Holm\thanks{Email address: \url{d.holm@imperial.ac.uk}}}
\author[2]{Lyle Noakes\thanks{Email address: \url{lyle.noakes@uwa.edu.au}}}
\author[1]{Joris Vankerschaver\thanks{Email address: \url{joris.vankerschaver@gmail.com}}}

\affil[1]{Department of Mathematics, Imperial College London, London SW7 2AZ, United Kingdom}
\affil[2]{School of Mathematics and Statistics, The University of Western Australia, 
35 Stirling Highway, Crawley WA 6009, Australia}

\DeclareMathOperator{\Cay}{Cay}
\DeclareMathOperator{\dCay}{dCay}
\DeclareMathOperator{\Ad}{Ad}

\def\bfi{\bf\it}

\def\lb{\left<}
\def\rb{\right>}

\def\llb{\left<\!\left<}
\def\rrb{\right>\!\right>}

\begin{document}
\maketitle

\begin{abstract}
We propose a notion of distance between two parametrized planar curves, called their \emph{discrepancy}, and defined intuitively as the minimal amount of deformation needed to deform the source curve into the target curve.  A precise definition of discrepancy is given as follows. A curve of transformations in the special Euclidean group $SE(2)$ is said to be \emph{admissible} if it maps the source curve to the target curve under the point-wise action of $SE(2)$ on the plane. After endowing the group $SE(2)$ with a left-invariant metric, we define a relative geodesic in $SE(2)$ to be a critical point of the energy functional associated to the metric, over all admissible curves. The discrepancy is then defined as the value of the energy of the minimizing relative geodesic. In the first part of the paper, we derive a scalar ODE which is a necessary condition for a curve in $SE(2)$ to be a relative geodesic, and we discuss some of the properties of the discrepancy. In the second part of the paper, we consider discrete curves, and by means of a variational principle, we derive a system of discrete equations for the relative geodesics. We finish with several examples.
\end{abstract}

\tableofcontents

\section{Introduction}

\paragraph{Historical overview.}
Computational anatomy is the modern study of anatomical shape and its variability. This study originated in 1917, in the seminal book \emph{Growth and Form} by D'Arcy Thompson, who recognized that anatomical comparison is a mathematical problem, and that its solution would lie in what he called \emph{the Theory of Transformations} \cite[p1032]{Th1942}.\footnote{In his book, Thompson says he ``learnt of it from Henri Poincar\'e''.} Since then, D'Arcy Thompson has been proven correct and many mathematical concepts, particularly  concepts from Lie groups, Riemannian geometry and analysis have been applied to solve what is now called ``the image registration problem''. An example is the comparison of medical images, which is a world-wide technology with an enormous number of uses every day. Perhaps not surprisingly, the comparison of medical images is still based on the Theory of Transformations, as enhanced by  its modern developments. 

Many mathematical frameworks have been developed to deal with the image registration problem, using the Theory of Transformations. The modern frameworks are well-developed and fascinating for the geometry and analysis that underlies their solutions of the image registration problem. An outstanding example is the framework of \emph{large deformation diffeomorphic metric mapping} (LDDMM).\footnote{Recall that a diffeomorphism is a smooth invertible map, with a smooth inverse.} Many papers and books have been written about LDDMM and its variants. We refer the reader to the fundamental texts \cite{MuDe2010,Yo2010}.

\paragraph{Contributions of this paper.}
The present work arose in the context of LDDMM, but it reformulates the problem of registration in a simple and direct manner, for the comparison of two planar curves without using the diffeomorphism group. The reformulation is based on the idea that the \emph{discrepancy} between two planar curves could be estimated quantitatively by defining the minimal amount of deformation needed to deform the source curve into the target curve using only the transformations of the special Euclidean group $SE(2)$ of rotations and translations acting point-wise on the plane. The point of the paper is to define \emph{discrepancy} precisely enough, that the quantitative comparison will be meaningful. 

Before explaining the main content of the paper and how it is organized, we present a few definitions and notation that set the context. 

Let $\mathbf{c}_0,\mathbf{c}_1:[0,1]\to M$ be $C^\infty$ curves, where $M$ is a $C^\infty$ manifold. Let $G$ be a Lie group with a Riemannian metric $\langle ~,~\rangle$, and a transitive left action $.$ on $M$. A $C^\infty$ function $g:[0,1]\to G$ is called {\em admissible} when $g(s).\mathbf{c}_0(s)=\mathbf{c}_1(s)$ for all $s\in [0,1]$.  The {\em energy} of an admissible $g$ is 
$$E_0(g)~:=~\frac{1}{2}\int _0^1\Vert g'(s)\Vert _{g(s)}^2~ds$$
where $\Vert ~\Vert _{g(s)}$ denotes  the Riemannian norm. A critical point $g$ of $E_0$ is said to be a {\em geodesic relative to} $\mathbf{c}_0,\mathbf{c}_1$, or just a {\em relative geodesic} when $\mathbf{c}_0,\mathbf{c}_1$ are understood. 

The present paper investigates relative geodesics in the special case where $G$ is the group $SE(2)$ of Euclidean motions of $\mathbb{R}^2$, with a particular kind of left-invariant Riemannian metric, and the standard left action on $M=\mathbb{R}^2$.   Then, given $\mathbf{c}_0,\mathbf{c}_1:[0,1]\to \mathbb{R}^2$, their {\em discrepancy} $\delta (\mathbf{c}_0,\mathbf{c}_1)$ is the minimum of $E_0(g)$ as $g$ varies over all admissible curves in $SE(2)$. A minimiser $g$ of $E_0$ is necessarily a relative geodesic although, as seen in Proposition \ref{prop:multiple_geodesics}, not all relative geodesics are minimisers. 

When the discrepancy is $0$ there is a constant admissible $g$, namely $\mathbf{c}_0,\mathbf{c}_1$ are {congruent}. More generally the minimum energy $\delta (\mathbf{c}_0,\mathbf{c}_1)$ is the minimum total variability of an admissible $g$. This is meant to capture, to some degree, the intuitive difficulty we find in transforming by eye from the parametric curve $\mathbf{c}_0$ to $\mathbf{c}_1$. 

Because the Riemannian metric on $SE(2)$ is chosen to be left-invariant, it follows easily that $\delta (\mathbf{c}_0,h.\mathbf{c}_1)=\delta (\mathbf{c}_0,\mathbf{c}_1)$ for any {constant} $h\in G$. Usually however $\delta (h.\mathbf{c}_0,\mathbf{c}_1)\not= \delta (\mathbf{c}_0,\mathbf{c}_1)$. Indeed usually $\delta (\mathbf{c}_0,\mathbf{c}_1)\not= \delta (\mathbf{c}_1,\mathbf{c}_0)$, as in Proposition \ref{prop:c0_zero}.  Sometimes there are continua of relative geodesics, as in Lemmas \ref{coroll_c0_zero}, \ref{coroll_c1_zero}. These facts are found by studying the Euler-Lagrange equations (called the continuous equations of motion) for relative geodesics, together with the natural boundary conditions at $s=0,1$. 
The Euler-Lagrange equations  for relative geodesics are derived in \S \ref{sec3} by Euler-Poincar\'e reduction, and also directly.  
 
In \S \ref{sec:discrete_relative_geodesics} we introduce discrete analogues of relative geodesics in $SE(2)$, as objects of separate interest, and 
 with a view to developing numerical methods for continuous relative geodesics. Discrete curves are finite sequences, and the discrete energy is defined in terms of the Cayley map for $SE(2)$. The  Euler-Poincar\'e approach adapts nicely to the discrete case, leading to projected discrete equations of motion (\ref{projected_eom}), and projected discrete boundary conditions (\ref{projected_bc}),  from which discrete relative geodesics can be found by Newton iteration. Alternatively, discrete relative geodesics can be calculated by direct minimisation of the discrete energy. 
 In \S \ref{sec:num} the numerical methods developed in \S \ref{sec:discrete_relative_geodesics} are used to illustrate properties of relative geodesics and discrepancies for some simple examples of plane curves.  A morphing procedure, using a minimal relative geodesic,  illustrates the geometrical difficulty of transforming $\mathbf{c}_0$ to  $\mathbf{c}_1$.

\section{The Euclidean Group $SE(2)$}

In this section, we recall some basic results about the Lie group $SE(2)$, which consists of rotations and translations in the Euclidean plane. For more information about $SE(2)$ and its role in mechanics and control, see \cite{MaRa1994, Ch2009, Ho2011}.

\paragraph{The Lie group $SE(2)$.}

The elements of $SE(2)$ are pairs
$g := (R_\theta, \mathbf{x})$ where $R_\theta \in SO(2)$ represents
the counterclockwise rotation over an angle $\theta \in \mathbb{S}^1$,
and $\mathbf{x} \in \mathbb{R}^2$ represents the translation along $\mathbf{x}$. There is a
one-to-one correspondence between elements $g = (R_\theta, \mathbf{x})
\in SE(2)$ and 3-by-3 matrices $\hat{g}$ of the form
\begin{equation} \label{matrix_group}
  \hat{g} = \begin{pmatrix}
		R_\theta & \mathbf{x} \\
		0 & 1
	\end{pmatrix}.
\end{equation}

In terms of matrices, the group multiplication in $SE(2)$ is given by matrix multiplication. In components, we have that $(R_\theta, \mathbf{x}) \cdot (R_\varphi, \mathbf{y}) = (R_{\theta+\varphi}, R_\theta \mathbf{x} + \mathbf{y})$ and $(R_\theta, \mathbf{x})^{-1} = (R_{-\theta}, - R_{-\theta} \mathbf{x})$.

\paragraph{The Lie algebra $\mathfrak{se}(2)$.}

The Lie algebra of $SE(2)$ will be denoted by $\mathfrak{se}(2)$. The elements of the 3-dimensional Lie algebra $\mathfrak{se}(2)$ can be viewed as infinitesimal rotations and translations in the plane, and are represented as vectors
\[
  \xi = 
    \begin{pmatrix}
      \omega \\ \mathbf{v}
    \end{pmatrix} \in \mathbb{R}^3,
\]
where $\omega \in \mathbb{R}$ and $\mathbf{v} \in \mathbb{R}^2$. There
is a one-to-one-correspondence between the elements $\xi$ of
$\mathfrak{se}(2)$ and 3-by-3 matrices $\hat{\xi}$ of the form
\begin{equation} \label{matrix_algebra}
    \hat{\xi} = 
		\begin{pmatrix}
		-\omega J & \mathbf{v} \\
		0 & 0 
		\end{pmatrix},
		\quad \text{where} \quad
		J = \begin{pmatrix} 
				0 & 1 \\
				-1 & 0 
			\end{pmatrix}.
\end{equation}
In terms of matrices, the Lie bracket on $\mathfrak{se}(2)$ is given by the matrix commutator: $[\hat{\xi}, \hat{\eta}] = \hat{\xi} \hat{\eta} - \hat{\eta} \hat{\xi}$, for all $\hat{\xi}, \hat{\eta} \in \mathfrak{se}(2)$.  In components, we have that $[(\omega, \mathbf{v}), (\eta, \mathbf{w})] = (0, -\omega J \mathbf{w} + \eta J \mathbf{v})$.

The elements of the dual space $\mathfrak{se}(2)^\ast$ are likewise column vectors in $\mathbb{R}^3$, denoted as 
\[
  \mu = 
    \begin{pmatrix} 
      \pi \\
      \mathbf{p}
    \end{pmatrix},
\]
with $\pi \in \mathbb{R}$ and $\mathbf{p} \in \mathbb{R}^2$. The
duality pairing is given in terms of the Euclidean inner product on
$\mathbb{R}^3$ by
\begin{equation} \label{pairing}
	\left< \mu, \xi \right> := \mu^T \xi = 
		\pi \omega + \mathbf{p}^T \mathbf{v}.
\end{equation}

\paragraph{Choice of a norm on $SE(2)$.}

Let $m > 0$ be a fixed parameter.  We define a norm on the vector space $\mathfrak{se}(2)$ by 
\begin{equation} \label{algebra_norm}
	\left\Vert \xi \right\Vert_{\mathfrak{se}(2)}^2 = \left\Vert (\omega, \mathbf{v}) \right\Vert_{\mathfrak{se}(2)}^2
		:= m \omega^2 + \mathbf{v}^T \mathbf{v},
\end{equation}
for $\xi = (\omega, \mathbf{v}) \in \mathfrak{se}(2)$. 

We extend this norm by left translation to a norm on the tangent vectors to the group $SE(2)$, given by 
\begin{equation} \label{group_norm}
	\left\Vert v_g \right\Vert_{SE(2)} = \left\Vert g^{-1} v_g \right\Vert_{\mathfrak{se}(2)}
\end{equation}
for all $v_g \in T_g SE(2)$. Here we represent $g$ in the matrix form \eqref{matrix_group} and the tangent vectors $v_g$ as 3-by-3 matrices 
\[
	v_g = \begin{pmatrix}
			- R_\theta J \theta' & \mathbf{x}' \\
			0 & 0
		\end{pmatrix},
\]
where $(\theta', \mathbf{x}') \in \mathbb{R} \times \mathbb{R}^2$. The multiplication on the right-hand side of \eqref{group_norm} is matrix multiplication, and $g^{-1} v_g \in \mathfrak{se}(2)$. Upon expanding the matrix product, we find for the norm
\begin{equation} \label{norm_explicit}
	\left\Vert v_g \right\Vert_{SE(2)}^2 = m (\theta')^2 + 
		\left\Vert \mathbf{x}' \right\Vert^2.	
\end{equation}
Note that the norm is by definition invariant with respect to the left action of $SE(2)$ on itself: 
\[
	\left\Vert h v_g \right\Vert_{SE(2)} = 
	\left\Vert (hg)^{-1} h v_g \right\Vert_{\mathfrak{se}(2)} = 
	\left\Vert g^{-1} g v_g \right\Vert_{\mathfrak{se}(2)} = 
	\left\Vert v_g \right\Vert_{SE(2)}.
\]

From now on, we will drop the subscripts on the norms just defined, and we will denote both by $\left\Vert \cdot \right\Vert$.

\paragraph{The action of $SE(2)$ on $\mathbb{R}^2$.}

The group $SE(2)$ acts on the plane $\mathbb{R}^2$ in the standard way: an element $(R_\theta, \mathbf{x}) \in SE(2)$ transforms a point $\mathbf{c} \in \mathbb{R}^2$ into $\mathbf{c}' = R_\theta \mathbf{c} + \mathbf{x}$. This action translates into an infinitesimal action of $\mathfrak{se}(2)$ on $\mathbb{R}^2$ defined by $(\omega, \mathbf{v}) \cdot \mathbf{c} = - \omega J \mathbf{c} + \mathbf{v}$. For $\mathbf{c} \in \mathbb{R}^2$ fixed, we denote by $\mathfrak{se}(2)_{\mathbf{c}}$ the {\bfi isotropy subalgebra} of elements in $\mathfrak{se}(2)$ that fix $\mathbf{c}$, that is,
\[
	\mathfrak{se}(2)_{\mathbf{c}} = 
		\{ (\omega, \mathbf{v}) \in \mathfrak{se}(2): \mathbf{v} = \omega J \mathbf{c} \}. 
\]

We let $\mathfrak{se}(2)^\circ_{\mathbf{c}}$ be the annihilator of $\mathfrak{se}(2)_{\mathbf{c}}$ in $\mathfrak{se}(2)^\ast$. In other words, $\mathfrak{se}(2)^\circ_{\mathbf{c}}$ consists of all covectors $(\pi, \mathbf{p})$ which vanish when contracted with the elements of $\mathfrak{se}(2)_{\mathbf{c}}$. A small calculation shows that
\begin{equation} \label{isotropy_annihilator}
	\mathfrak{se}(2)^\circ_{\mathbf{c}} = 
		\{ (\pi, \mathbf{p}) \in \mathfrak{se}(2)^\ast: 
			\pi = - \mathbf{p} \cdot J \mathbf{c} \}
\end{equation}

Note that $\mathfrak{se}(2)_{\mathbf{c}}$ is isomorphic to $\mathfrak{so}(2)$, while $\mathfrak{se}(2)^\circ_{\mathbf{c}}$ is isomorphic to $\mathbb{R}^2$.

Lastly, we define the projection $\mathbb{P}_{\mathbf{c}}: \mathbf{se}(2)^\ast \to \mathbb{R}$ by
\begin{equation} \label{annih_projection}
	\mathbb{P}_{\mathbf{c}}(\pi, \mathbf{p}) = \pi + \mathbf{p} \cdot J \mathbf{c},
\end{equation}	
so that $\mathfrak{se}(2)^\circ_{\mathbf{c}}$ is precisely the kernel of $\mathbb{P}_{\mathbf{c}}$. This map will be useful later on.

\section{Continuous Relative Geodesics in $SE(2)$}\label{sec3}

Throughout this section, we let $\mathbf{c}_0, \mathbf{c}_1 : [0, 1] \to \mathbb{R}^2$ be two fixed parametrized curves. We say that a curve $g: [0, 1] \to SE(2)$ is {\bfi admissible} with respect to $\mathbf{c}_0$ and $\mathbf{c}_1$ if 
\[
	g(s) \cdot \mathbf{c}_0(s) = \mathbf{c}_1(s), 
		\quad
			\text{for all $s \in [0, 1]$},
\]
where the dot on the left hand side represents the standard action of $SE(2)$ on $\mathbb{R}^2$. In other words, a curve $g(s) = (R_{\theta(s)}, \mathbf{x}(s))$ is admissible if 
\begin{equation} \label{se2_constraint}
	R_{\theta(s)} \mathbf{c}_0(s) + \mathbf{x}(s) = \mathbf{c}_1(s),
	\quad
			\text{for all $s \in [0, 1]$}.
\end{equation}

\subsection{The Deformation Energy}

We now wish to find the admissible curves in $SE(2)$ which minimize the {\bfi deformation energy} 
\begin{equation} \label{def_energy_group}
	E_0 = \frac{1}{2} \int_0^1 \left\Vert g'(s) \right\Vert^2 \, ds,
\end{equation}
where the norm $\left\Vert \cdot \right\Vert$ was given in \eqref{group_norm}, and the prime $'$ represents the derivative with respect to the curve parameter $s$. The deformation energy measures the change in $g(s)$ as $s$ varies. 

\begin{definition}
	Let $\mathbf{c}_0, \mathbf{c}_1 : [0, 1] \to \mathbb{R}^2$ be two parametrized curves. A admissible curve $g: [0, 1] \to SE(2)$ with respect to $\mathbf{c}_0, \mathbf{c}_1$ is a {\bfi relative geodesic} if is a minimum of the deformation energy $E_0$ over all admissible curves.  
\end{definition} 

Since the norm \eqref{group_norm} is invariant with respect to the multiplication from the left by elements of $SE(2)$, we may write the deformation energy equivalently as 
\begin{align}
	E_0 & = \frac{1}{2} \int_0^1 \left\Vert g^{-1}(s) g'(s) \right\Vert^2 \, ds 
		\nonumber \\
	& = \frac{1}{2} \int_0^1 m \omega(s)^2 + 
		\left\Vert \mathbf{v}(s) \right\Vert^2 \, ds,  \label{energy0}
\end{align}
where in the second expression the definition  \eqref{algebra_norm} has been used.  Here, $(\omega(s), \mathbf{v}(s)) \in \mathfrak{se}(2)$ in the Lie algebra is related to the group element $g(s) = (R_{\theta(s)}, \mathbf{x}(s))$ by means of the equations
\begin{equation} \label{reconstr_eq}
	\omega = \theta' 
		\quad \text{and} \quad 
	\mathbf{v} = R_{-\theta} \mathbf{x}',
\end{equation}
which follow from expanding the left-trivialized derivative $g^{-1}(s) g'(s)$: 
\[
	g^{-1}(s) g'(s) = 
	\begin{pmatrix}
		R_{-\theta} & - R_{-\theta} \mathbf{x} \\
		0 & 1
	\end{pmatrix}
	\begin{pmatrix}
		-R_\theta J \theta'  & \mathbf{x}' \\
		0 & 0
	\end{pmatrix} = 
	\begin{pmatrix}
		-J \theta' & R_{-\theta} \mathbf{x}' \\
		0 & 0
	\end{pmatrix}.
\]
Using the identification \eqref{matrix_algebra}, we then arrive at the equations \eqref{reconstr_eq}. These relations are referred to as the {\bfi reconstruction relations}: given $\omega(s)$ and $\mathbf{v}(s)$, \eqref{reconstr_eq} may be viewed as a set of first-order ODEs specifying the components of $g(s)$.

For the purpose of deriving the equations that determine the extremals of $E_0$, it will be convenient to add the reconstruction relations as constraints to the deformation energy, so that we obtain
\begin{equation} \label{def_energy_se}
	E =  \int_0^1 \frac{1}{2} \left( m \omega^2(s) + 
		\left\Vert \mathbf{v}(s) \right\Vert^2 \right)
	+ \pi(s)(\theta'(s) - \omega(s))
	+ \mathbf{p}(s)^T(R_{-\theta(s)} \mathbf{x}'(s) - \mathbf{v}(s)) \, ds.
\end{equation}
Here, $E$ depends now on the curve $(R_{\theta}, \mathbf{x})$, the Lie algebra elements $(\omega, \mathbf{v}) \in \mathfrak{se}(2)$, and the Lagrange multipliers $(\pi, \mathbf{p})$, which can be viewed as elements of the dual $\mathfrak{se}(2)^\ast$ of the Lie algebra.  It will be shown below that the critical points of this augmented functional coincide with the critical points of the original deformation energy, given in \eqref{def_energy_group}.

Note that $E$ can be written in a more concise, Lie-algebraic way as 
\begin{equation} \label{def_energy}
	E =  \int_0^1 \frac{1}{2} \llb \xi(s), \xi(s) \rrb + \lb \mu(s), g^{-1}(s) g'(s) - \xi(s) \rb \, ds,
\end{equation}
where $\xi(s) = (\omega(s), \mathbf{v}(s)) \in \mathfrak{se}(2)$, $\mu(s) = (\pi(s), \mathbf{p}(s)) \in \mathfrak{se}(2)^\ast$, and the brackets $\llb \cdot, \cdot \rrb$ refer to the inner product associated to the norm \eqref{algebra_norm}. This energy function can be generalized in a straightforward way to the case of relative geodesics with values in an arbitrary Lie group $G$.  

The variational principle for \eqref{def_energy}, in which the configuration variables, velocities and momenta are varied independently while the reconstruction equations are treated as constraints, is a particular example of the \emph{Hamiltonian-Pontryagin principle} (see \cite{YoMa2006b}). A version of the Hamilton-Pontryagin principle specific to Lie groups can be found in \cite{CeMaPeRa2003}; see also \cite{BoMa2009}.  Our variational principle is also related to the \emph{Clebsch variational principle} of \cite{CoHo2009, GaRa2011}, although it does not coincide with it.

\subsection{The Continuous Equations of Motion}

We now derive the differential equations that describe the critical points of the deformation energy. Because of the analogy with the Euler-Lagrange equations in mechanics, we will refer to these equations as equations of motion.

We take variations of the augmented deformation energy $E$ in \eqref{def_energy_se} with respect to the variables $(\theta, \mathbf{x}, \omega, \mathbf{v}, \pi, \mathbf{p})$, where the velocities $(\omega, \mathbf{v})$ and the momenta $(\pi, \mathbf{p})$ are varied freely, while the configuration variables $(\theta, \mathbf{x})$ are varied with respect to variations that preserve the admissibility constraint \eqref{se2_constraint}. On the level of the variations, the infinitesimal version of this constraint is given by 
\begin{equation} \label{continuous_inf_constraint}
	- R_{\theta} J \mathbf{c}_0 \, \delta \theta + \delta \mathbf{x} = 0,
\end{equation}
where the matrix $J$ was given in \eqref{matrix_algebra}, and this relation allows us to eliminate the variation $\delta \mathbf{x}$ in terms of $\delta \theta$. This infinitesimal constraint was obtained by taking a one-parameter family $(\theta_\epsilon(s), \mathbf{x}_\epsilon(s))$ of admissible curves in $SE(2)$: taking the derivative of the admissibility constraint \eqref{se2_constraint} with respect to $\epsilon$, and putting 
\[
	\delta \theta = \frac{d \theta_\epsilon}{d \epsilon} \Big|_{\epsilon=0} 
		\quad \text{and} \quad
	\delta \mathbf{x} = \frac{d \mathbf{x}_\epsilon}{d \epsilon} \Big|_{\epsilon=0} 	
\]
then gives \eqref{continuous_inf_constraint}.

Taking variations of $E$ first with respect to $\pi$ and $\mathbf{p}$ results in 
\[
	D E \cdot \delta \pi = \int_0^1 \delta\pi(s)(\theta'(s) - \omega(s)) ds
\]
and
\[
	D E \cdot \delta \mathbf{p} = 
		\int_0^1  \delta \mathbf{p}(s)^T (
			R_{-\theta(s)} \mathbf{x}'(s) - \mathbf{v}(s)) \, ds,
\]
so that if a curve $s \mapsto (\theta(s), \mathbf{x}(s), \omega(s), \mathbf{v}(s), \pi, \mathbf{p}(s))$ is a critical point of $E$ then the reconstruction equations \eqref{reconstr_eq} must hold. Taking variations with respect to the velocities $\omega$ and $\mathbf{v}$ similarly results in 
\[
	D E \cdot \delta \omega = \int_0^1 ( m \omega(s) - \pi(s) ) \delta \omega(s) \, ds
\]
and 
\[
	D E \cdot \delta \mathbf{v} = \int_0^1 (\mathbf{v}(s) - \mathbf{p}(s)) 
		\cdot \mathbf{v}(s) \, ds,
\]
so that for a critical point of $E$ the following Legendre transformations between the velocities and the momenta must hold:
\begin{equation} \label{continuous_leg_trafo}
	\mathbf{p} = \mathbf{v} \quad \text{and} \quad \pi = m \omega.
\end{equation}

Lastly, taking variations with respect to the configuration variables $(\theta, \mathbf{x})$ we obtain 
\begin{align*}
	\delta E = & 
		\int_0^1 \left[
		 \left( - \frac{d\pi}{ds} + \mathbf{p}^T R_{-\theta}J\mathbf{x}' \right)
		\delta \theta - 
		\frac{d}{ds} (R_\theta \mathbf{p})^T \delta \mathbf{x} \right] \, ds
			\\
		& + \Big( \pi \delta \theta + 
			(R_\theta \mathbf{p})^T \delta \mathbf{x}
								\Big)\Big|_{s = 0, 1},
\end{align*}
where we have integrated by parts. We now use \eqref{continuous_inf_constraint} to eliminate $\delta \mathbf{x}$ in function of $\delta \theta$, and we obtain 
\begin{align*}
	\delta E = & 
		\int_0^1 \left( - \frac{d\pi}{ds} + \mathbf{p}^T R_{-\theta}J\mathbf{x}'
			- \frac{d}{ds} (R_\theta \mathbf{p})^T R_\theta J \mathbf{c}_0 \right)
						\delta \theta  \, ds 
				\\
			& + \left( \pi + \mathbf{p}^T J \mathbf{c}_0 \right) \delta \theta 
					\big|_{s=0,1}.
\end{align*} 
Since $\delta \theta$ is arbitrary, we see that $\delta E$ vanishes whenever the expressions preceding $\delta \theta$ on the right-hand side vanish, so that 
\begin{equation} \label{continuous_eq_se2}
	- \frac{d\pi}{ds} + \mathbf{p}^T R_{-\theta}J\mathbf{x}'
			- \frac{d}{ds} (R_\theta \mathbf{p})^T R_\theta J \mathbf{c}_0 = 0,
\end{equation}
with boundary conditions
\[
	\pi + \mathbf{p}^T J \mathbf{c}_0 = 0 \quad \text{for $s = 0, 1$}.
\]
Note that in these equations, $\theta$ and $\mathbf{x}$ are not arbitrary, but are related by the admissibility constraint \eqref{se2_constraint}. By writing \eqref{continuous_eq_se2} in terms of the Lie algebra quantities $(\omega, \mathbf{v})$, and with some simple algebraic simplifications, we finally arrive at the following result.

\begin{theorem}\label{thmep}
	Let $\mathbf{c}_0, \mathbf{c}_1 : [0, 1] \to \mathbb{R}^2$ be two parametrized curves. An admissible curve $(R_{\theta(s)}, \mathbf{x}(s))$ in $SE(2)$ is a critical point of the deformation energy $E$  
	if and only if the following scalar equation holds:
	\begin{equation} \label{SE_EOM}
		m \omega' - \mathbf{c}_0^T J \mathbf{v}' - 
			\mathbf{c}_0^T \mathbf{v} \theta' = 0,
	\end{equation}
	with natural boundary conditions $m \omega + \mathbf{v}^T J \mathbf{c}_0 = 0$ for $s = 0, 1$.  Here, $\omega$ and $\mathbf{v}$ are given in terms of $\theta$ and $\mathbf{x}$ by the reconstruction relations \eqref{reconstr_eq}, and the admissibility constraint \eqref{se2_constraint} holds.
\end{theorem}

As the relative geodesics are precisely the minima of the deformation energy $E$, restricted to the space of admissible curves, the equation \eqref{SE_EOM} is a necessary condition for a curve $g: [0, 1] \to SE(2)$ to be a relative geodesic.

The equation of motion \eqref{SE_EOM} can be written in a form which involves the angle $\theta$ only. To this end, we differentiate the admissibility constraint \eqref{se2_constraint} with respect to $s$ and multiply from the left by $R_{-\theta}$ to obtain 
\[
	\mathbf{v} = R_{-\theta} \mathbf{c}_1' - \mathbf{c}_0' 
		+ J \mathbf{c}_0 \theta'.
\]
One further differentiation leads to 
\[
	\mathbf{v}' = R_{-\theta} J \mathbf{c}_1' \theta' + R_{-\theta} \mathbf{c}_1''
		- \mathbf{c}_0'' + J \mathbf{c}_0' \theta' + J \mathbf{c}_0 \theta''.
\]
Substituting these two relations into \eqref{SE_EOM} leads after some simplifications to the following non-autonomous second-order ODE for $\theta$:
\[
	( m + \left\Vert \mathbf{c}_0 \right\Vert^2 ) \theta'' 
		+ 2 \mathbf{c}_0^T \mathbf{c}_0' \theta' 
		+ \mathbf{c}_0^T J(  \mathbf{c}_0'' - R_{-\theta} \mathbf{c}_1'') = 0,
\]
with boundary conditions 
\[
	(m + \left\Vert \mathbf{c}_0 \right\Vert^2) \theta' + 
		\mathbf{c}_0^T J (\mathbf{c}_0' - R_{-\theta} \mathbf{c}_1') = 0
		 \quad \text{for $s = 0, 1$}.
\]

Given the two curves $\mathbf{c}_0$ and $\mathbf{c}_1$, the previous equations form a boundary-value problem for $\theta$. Once $\theta$ is determined from these equations, the linear displacement $\mathbf{x}$ can be found from the admissibility constraint \eqref{se2_constraint}.

\paragraph{A direct derivation of the equations of motion.} 

In this section we present an alternative derivation of the equations \eqref{SE_EOM}, which does not use the Euler-Poincar\'e framework. This derivation is arguably somewhat more straightforward than the one presented earlier, and we will use the resulting Euler-Lagrange equations extensively in Section~\ref{sec:discrepancy} below. The advantage of the Euler-Poincar\'e equations, however, is that they can easily be discretized, as we shall show in Section~\ref{sec:discrete_relative_geodesics}.

We begin by introducing the function $\mathcal{E}_0$ given by
\[
	\mathcal{E}_0 = 
		\frac{m}{2} ( \theta' )^2 + 
                \frac{1}{2} \left\Vert \mathbf{x}' \right\Vert^2.
\]
Note that $\mathcal{E}_0$ is precisely the integrand of the deformation energy $E_0$ in \eqref{energy0}. We now take the derivative of the admissibility constraint \eqref{se2_constraint}, and use the resulting equation to obtain an expression for $\mathbf{x}'$. Upon substituting this expression into $\mathcal{E}_0$, we obtain a function $\ell$ which depends on $\theta$ and $\theta'$, and is given by
\begin{equation} \label{lagrangian}
	\ell = 
		\frac{m}{2} ( \theta' )^2 + 
		\frac{1}{2}\left\Vert R_{-\theta} \mathbf{c}_1' - \mathbf{c}_0' + 
			J \mathbf{c}_0 \theta'  \right\Vert^2.
\end{equation}

The function $\ell$ can now be viewed as a Lagrangian function on the tangent bundle $T \mathbb{S}^1$; its Euler-Lagrange equations are 
\begin{equation} \label{EL_general}
	\frac{d}{ds} \left( \frac{\partial \ell}{\partial \theta'} \right) - 
		\frac{\partial \ell}{\partial \theta} = 0
\end{equation}
with natural boundary conditions
\[
	\frac{\partial \ell}{\partial \theta'}  = 0, \quad \text{for $s = 0, 1$}.
\]

For further reference, we define the momentum conjugate to $\theta$ as 
\begin{equation} \label{scalar_p}
	p := \frac{\partial \ell}{\partial \theta'} = 
		( m + \left\Vert \mathbf{c}_0 \right\Vert^2 ) \theta' 
		+ (\mathbf{c}_1')^T R_\theta J \mathbf{c}_0 - (\mathbf{c}_0')^T J \mathbf{c}_0,
\end{equation}
and compute 
\[
	\frac{\partial \ell}{\partial \theta} = 
		(\mathbf{c}_0)^T R_{-\theta} \mathbf{c}_1' \theta' 
		- (\mathbf{c}_0')^T R_{-\theta} J \mathbf{c}_1'.
\]

By substituting these expressions into the Euler-Lagrange equations \eqref{EL_general}, we obtain yet another set of equations characterizing relative geodesics, which we summarize in the following theorem.

\begin{theorem} \label{thm:euler_lagrange}
Let $\mathbf{c}_0, \mathbf{c}_1 : [0, 1] \to \mathbb{R}^2$ be two parametrized curves. An admissible curve $(R_{\theta(s)}, \mathbf{x}(s))$ in $SE(2)$ is a critical point of the deformation energy $E_0$  if and only if the following Euler-Lagrange equation for $\theta$ holds:
\begin{equation} \label{euler_lagrange}
	\frac{d p}{d s} = (\mathbf{c}_0)^T R_{-\theta} \mathbf{c}_1' \theta' 
		- (\mathbf{c}_0')^T R_{-\theta} J \mathbf{c}_1',
\end{equation}
with natural boundary conditions given by $p = 0$ for $s = 0, 1$.
Here $p(s)$ is given by \eqref{scalar_p} and $\mathbf{x}(s)$ is expressed as a function of $\theta(s)$ using the admissibility constraint \eqref{se2_constraint}.
\end{theorem}

The procedure of substituting the constraints into the deformation energy to obtain a Lagrangian function which depends on fewer degrees of freedom is similar to the approach of Chaplygin for systems with nonholonomic kinematic constraints (see \cite{Ko1992} and the references therein). In this approach, one eliminates the constrained degrees of freedom to obtain a system of reduced Euler-Lagrange equations with gyroscopic forces. The latter vanish if the constraints are integrable, as is the case for relative geodesics.

\subsection{Discrepancy between Planar Curves} \label{sec:discrepancy}

\begin{definition}
Let $\mathbf{c}_0, \mathbf{c}_1: [0, 1] \to \mathbb{R}^2$ be two parametrized curves in the plane. The \emph{discrepancy} $\delta(\mathbf{c}_0, \mathbf{c}_1)$ between $\mathbf{c}_0$ and $\mathbf{c}_1$ is the minimum of the deformation energy $E_0$ over all $(\mathbf{c}_0, \mathbf{c}_1)$-admissible curves:
\[
  \delta(\mathbf{c}_0, \mathbf{c}_1) = 
  \min_{g(\cdot) \in \mathrm{Adm}(\mathbf{c}_0, \mathbf{c}_1)} E_0(g'),
\]
where $\mathrm{Adm}(\mathbf{c}_0, \mathbf{c}_1)$ is the set of all $(\mathbf{c}_0, \mathbf{c}_1)$-admissible curves.
\end{definition}

Note that the admissible curve $g: [0, 1] \to SE(2)$ which minimizes $E_0$ can be found among the solutions of the equations of motion derived in the previous section.

\paragraph{Asymmetry of the discrepancy.}

The discrepancy $\delta(\mathbf{c}_0, \mathbf{c}_1)$ provides a measure of the difference between the curves $\mathbf{c}_0$ and $\mathbf{c}_1$. In this section, we show that the discrepancy is in general not symmetric, that is, $\delta(\mathbf{c}_0, \mathbf{c}_1)$ differs in general from $\delta(\mathbf{c}_1, \mathbf{c}_0)$. 

Throughout the remainder of this section, we use the formulation of the Euler-Lagrange equations given in Theorem~\ref{thm:euler_lagrange}.

\begin{lemma} \label{coroll_c0_zero}
Let $\mathbf{c}_0(s) = \mathbf{0}$ for all $s$. The geodesics relative to $(\mathbf{c}_0, \mathbf{c}_1)$ are parametrized by $\theta(0)$, and in each case
\[
	\delta( \mathbf{c}_0, \mathbf{c}_1 ) = \frac{1}{2}
		\int_0^1 \left\Vert \mathbf{c}_1' \right\Vert^2 \, ds.
\]
\end{lemma}
\begin{proof}
For a geodesic relative to $(\mathbf{c}_0, \mathbf{c}_1)$ we have that $p = m \theta'$, and from the Euler-Lagrange equations \eqref{euler_lagrange} it follows that $m \theta'' = 0$, so that $\theta$ is an affine function of $s$. By using the natural boundary conditions, it follows that $\theta$ is constant. Conversely any constant $\theta$ satisfies the Euler-Lagrange equations \eqref{euler_lagrange} and the natural boundary conditions, and therefore defines a relative geodesic.
\end{proof}

\begin{lemma} \label{coroll_c1_zero}
Let $\mathbf{c}_1(s) = \mathbf{0}$ for all $s$. The geodesics relative to $(\mathbf{c}_0, \mathbf{c}_1)$ are parametrized by $\theta(0)$, and in each case the discrepancy is given by
\begin{equation} \label{discrep_c1_zero}
	\delta( \mathbf{c}_0, \mathbf{c}_1 ) = \frac{1}{2}
		\int_0^1 \left( \left\Vert \mathbf{c}_0' \right\Vert^2 
			- \frac{((\mathbf{c}_0')^T J \mathbf{c}_0)^2}%
				{m + \left\Vert \mathbf{c}_0 \right\Vert^2}
		\right) \, ds.
\end{equation}
\end{lemma}
\begin{proof}
By \eqref{euler_lagrange}, $p$ is constant, and by the natural boundary conditions, $p$ is identically $0$, so that
\[
	\theta' = \frac{(\mathbf{c}_0')^T J \mathbf{c}_0}%
		{m + \left\Vert \mathbf{c}_0 \right\Vert^2}.
\]
By substituting this expression for $\theta'$ into the deformation energy, we obtain \eqref{discrep_c1_zero} after some simplifications. 
\end{proof}

\begin{proposition} \label{prop:c0_zero}
Let $\mathbf{c}_0(s) = \mathbf{0}$ for all $s$. Then $\delta(\mathbf{c}_1, \mathbf{c}_0) \le \delta(\mathbf{c}_0, \mathbf{c}_1)$, with equality only in the case where, for some $\mathbf{x}_0 \in \mathbb{R}^2$ and some $\phi: [0, 1] \to \mathbb{R}$, $\mathbf{c}_1(s) = \phi(s) \mathbf{x}_0$.
\end{proposition}
\begin{proof}
Combining lemmas \ref{coroll_c0_zero} and \ref{coroll_c1_zero}, we have that
\[
\delta(\mathbf{c}_1, \mathbf{c}_0 ) = \delta( \mathbf{c}_0, \mathbf{c}_1 )
  - \int_0^1 \frac{((\mathbf{c}_1')^T J \mathbf{c}_1)^2}%
	{m + \left\Vert \mathbf{c}_1 \right\Vert^2} \, ds,
\]
so that $\delta( \mathbf{c}_1, \mathbf{c}_0 ) < \delta( \mathbf{c}_0, \mathbf{c}_1 )$ except when $(\mathbf{c}_1')^T J \mathbf{c}_1$ vanishes identically. In this case, $\mathbf{c}_1'(s) = \mu(s) \mathbf{c}_1(s)$ for some function $\mu$, and therefore $\mathbf{c}_1(s) = \phi(s)\mathbf{x}_0$, with $\phi(s) = e^{\mu(s)}$.
\end{proof}

As an illustration, we take $\mathbf{c}_0(s) = \mathbf{0}$ and $\mathbf{c}_1(s) = (\cos(\pi s), \sin(\pi s) )$ for $s \in [0, 1]$. By Lemmas~\ref{coroll_c0_zero} and \ref{coroll_c1_zero}, we have that
\[
\delta(\mathbf{c}_0, \mathbf{c}_1) = \frac{\pi^2}{2} \approx 4.93, \quad \text{and} \quad
\delta(\mathbf{c}_1, \mathbf{c}_0) = \frac{1}{2} \frac{m}{m+1} \pi^2,
\]
so that indeed $\delta(\mathbf{c}_0, \mathbf{c}_1) > \delta(\mathbf{c}_1, \mathbf{c}_0)$.

\paragraph{Non-minimising relative geodesics.}

There are always at least two geodesics relative to $(\mathbf{c}_0, \mathbf{c}_1)$, corresponding to critical points of the deformation energy $E_0$ regarded as a function of $\theta_0$ where the relative geodesic $s \mapsto (R_{\theta(s)}, \mathbf{x}(s))$ satisfies \eqref{euler_lagrange} for all $s$ and the natural boundary conditions at $s = 0$. 

\begin{proposition} \label{prop:multiple_geodesics}
Let $\mathbf{c}_1$ be a nonconstant affine line segment, and suppose $\mathbf{c}_0(0) = \mathbf{0}$ with $\mathbf{c}_0(1) \ne \mathbf{0}$. Then there are exactly two geodesics relative to $(\mathbf{c}_0, \mathbf{c}_1)$, and these are determined by 
\[
\theta(s) = \int_s^1 \frac{(\mathbf{c}_0'(u))^T J \mathbf{c}_0(u)}%
      {m + \left\Vert \mathbf{c}_0(u) \right\Vert^2} \, du \pm \chi_1
\]
with $\chi_1$ the angle between $\mathbf{c}_0(1)$ and $\mathbf{c}_1'(1)$. Only one of these relative geodesics is a global minimiser of $E_0$.
\end{proposition}
\begin{proof}
Since $\mathbf{c}_1'(s) =: \mathbf{C}$, where $\mathbf{C}$ is a constant vector, the right-hand side of the Euler-Lagrange equations \eqref{euler_lagrange} can be written as a total $s$-derivative:
\[
(\mathbf{c}_0)^T R_{-\theta} \mathbf{c}_1' \theta' 
		- (\mathbf{c}_0')^T R_{-\theta} J \mathbf{c}_1' = 
          - \frac{d}{ds} \left( \mathbf{c}_0^T R_{-\theta} J \mathbf{C}\right).
\]
Consequently, the Euler-Lagrange equations imply that the quantity
\begin{equation} \label{perturbed_momentum}
  \hat{p} := p + \mathbf{c}_0^T R_{-\theta} J \mathbf{C} 
    = ( m + \left\Vert \mathbf{c}_0 \right\Vert^2 ) \theta' 
      - (\mathbf{c}_0')^T J \mathbf{c}_0
\end{equation}
is conserved. Since $\mathbf{c}_0(0) = \mathbf{0}$ and $p = 0$ at $s = 0$, $\hat{p}(s)$ is identically 0. So, for some $\theta_0$, 
\[
  \theta(s) = \theta_0 - \int_0^s \frac{(\mathbf{c}_0'(u))^T J \mathbf{c}_0(u)}%
      {m + \left\Vert \mathbf{c}_0(u) \right\Vert^2} \, du.
\]
At the terminal end of the curve, i.e. for $s = 1$, we have that $p(1) = \hat{p}(1) = 0$. From \eqref{scalar_p} it then follows that $p(1) - \hat{p}(1) = \mathbf{C}^T R_{\theta(1)} J \mathbf{c}_0(1) = 0$, namely $R_{\theta(1)} \mathbf{c}_0(1)$ is a multiple of $\mathbf{C}$, so that $\theta(1) = \pm \chi_1$ for some angle $\chi_1$. As a consequence, the initial angle $\theta_0$ satisfies
\[
  \theta_0 = \int_0^1 \frac{(\mathbf{c}_0'(u))^T J \mathbf{c}_0(u)}%
      {m + \left\Vert \mathbf{c}_0(u) \right\Vert^2} \, du \pm \chi_1.
\]
Considering $E_0$ as a function of $e^{\mathrm{i} \theta_0} \in \mathbb{S}^1$, one of these values of $\theta_0$ is a point of global maximum, the other is a point of global minimum, and $E_0$ has no other critical points.
\end{proof}

As an illustration, we consider the discrepancy between two line segments. We take $\mathbf{c}_0(s) = s \mathbf{e}_x$ and $\mathbf{c}_1(s) = s \mathbf{e}_y$, with $\mathbf{e}_x, \mathbf{e}_y$ the standard unit vectors along the positive $x$- and $y$-axis, respectively.  From the previous proposition, we deduce that $\theta(s) = \pm \pi/2$ for all $s$. 

For the solution with $\theta(s) = \pi/2$, the admissibility condition results in $\mathbf{x}(s) = 0$ for all $s$. In this case, the effect of applying the relative geodesic is to rotate all of the points of $\mathbf{c}_0$ over $\pi/2$, and not effect any translation. As the relative geodesic is constant, $g(s) = (R_{\pi/2}, \mathbf{0})$, the deformation energy vanishes identically, so that $g(s)$ is a minimizing geodesic. In the case where $\theta(s) = - \pi/2$, the admissibility constraint yields $\mathbf{x}(s) = 2 s \mathbf{e}_y$ so that the effect of the relative geodesic is to rotate each point $\mathbf{c}_0(s)$ over $-\pi/2$, followed by a translation over $2 s \mathbf{e}_y$. The deformation energy in this case is $E = 2$.

\begin{remark}
In the proof of Proposition~\ref{prop:multiple_geodesics}, we have seen that the quantity $\hat{p}$ is conserved when $\mathbf{c}_1'(s) = \mathbf{C}$ with $\mathbf{C}$ constant. A natural question to ask is the following: Is there a continuous symmetry whose associated conserved quantity (through Noether's theorem) is precisely $\hat{p}$? To see that this is indeed the case, we return to the Lagrangian $\ell$ in \eqref{lagrangian} which we rewrite as
\begin{align*}
  \ell & = \frac{m}{2} ( \theta' )^2 + 
      \frac{1}{2} \left( \left\Vert \mathbf{c}_1' \right\Vert^2 
        + \left\Vert \mathbf{c}_0' - J \mathbf{c}_0 \theta' \right\Vert^2\right)
      + (\mathbf{c}_1')^T R_\theta (J \mathbf{c}_0 \theta' - \mathbf{c}_0') \\
      & = \hat{\ell} - \frac{d}{ds}\left( (\mathbf{c}_1')^T R_\theta \mathbf{c}_0 \right),
\end{align*}
where $\hat{\ell}$ is defined as 
\begin{equation} \label{hat-ell}
  \hat{\ell} = \frac{m}{2} ( \theta' )^2 + 
      \frac{1}{2} \left( \left\Vert \mathbf{c}_1' \right\Vert^2 
        + \left\Vert \mathbf{c}_0' - J \mathbf{c}_0 \theta' \right\Vert^2\right)        + (\mathbf{c}_1'')^T R_\theta \mathbf{c}_0.
\end{equation}

Since $\ell$ and $\hat{\ell}$ differ by a total $s$-derivative, they give rise to the same Euler-Lagrange equations (see \cite{Ol1986}). For $\hat{\ell}$, the momentum conjugate to $\theta$ is precisely the quantity $\hat{p}$ defined in \eqref{perturbed_momentum}: 
\[
  \hat{p} = \frac{\partial \hat{\ell}}{\partial \theta'} 
    = (m + \left\Vert \mathbf{c}_0 \right\Vert^2 ) \theta' - (\mathbf{c}_0')^T J \mathbf{c}_0.
\]
In the case that $\mathbf{c}_1'' = 0$, we see from \eqref{hat-ell} that $\hat{\ell}$ does not depend on $\theta$, so that $\hat{p}$ is a conserved quantity:
\[
  \frac{d \hat{p}}{ds} = \frac{\partial \hat{\ell}}{\partial \theta} = 0 
      \quad \quad \text{(when $\mathbf{c}_1'' = 0$)}.
\]
\end{remark}

\section{Discrete Relative Geodesics in $SE(2)$} \label{sec:discrete_relative_geodesics}

We now assume we have two discrete curves $(\mathbf{c}_0)_k,
(\mathbf{c}_1)_k$, $k = 0, \ldots, N$ of $N$ points each.  We wish to find a discrete curve $g_k = (R_{\theta_k}, \mathbf{x}_k)$, $k = 0, \ldots, N$ in $SE(2)$ which is {\bfi admissible} in the sense that 
\begin{equation} \label{discrete_admissibility}
	R_{\theta_k} (\mathbf{c}_0)_k + \mathbf{x}_k = (\mathbf{c}_1)_k
\end{equation}
for all $k = 0, \ldots, N$.  To derive a discrete version of the deformation energy $E$, we need to discretize the spatial derivatives that appear in \eqref{def_energy}. We do this by means of the Cayley map from $\mathfrak{se}(2)$ to $SE(2)$.

Our way of discretizing the variational principle, as well as the discrete equations obtained from it, is inspired by the \emph{discrete Hamilton-Pontryagin principle} of \cite{BoMa2009}; see also \cite{KoMa2011} and \cite{St2010}. As in the continuous case, the main difficulty here is the incorporation of the admissibility constraint \eqref{discrete_admissibility}.

\subsection{The Cayley Map}

\paragraph{Definition.}
The Cayley map
$\mathrm{Cay}: \mathfrak{se}(2) \to SE(2)$ is given by
\begin{equation} \label{cayley_se2}
	\mathrm{Cay}(\xi) = 
		\begin{pmatrix} 
			\hat{R}_{\omega} & \hat{\mathbf{x}}_\xi \\
			0 & 1
		\end{pmatrix},
\end{equation}
where, if $\xi = (\omega, \mathbf{v}) \in \mathfrak{se}(2)$,
\begin{equation} \label{sub_matrices}
	\hat{R}_\omega := 
	\frac{1}{1 + \omega^2/4}
		\begin{pmatrix}
			1 - \omega^2/4 & -\omega \\
			\omega & 1 - \omega^2/4  \\	
		\end{pmatrix},
	\quad \text{and} \quad
	\hat{\mathbf{x}}_\xi := \frac{1}{1 + \omega^2/4}
		\begin{pmatrix}
			v_1-\omega v_2/2 \\
			v_2+\omega v_1/2 
		\end{pmatrix},
\end{equation}
with $\mathbf{v} = (v_1, v_2)$.  Note that $\hat{R}_\omega$ depends
only on $\omega$ and is in fact the Cayley transform in $SO(2)$.

The Cayley map is in fact a $(1, 1)$-Pad\'e approximation to the exponential map from $\mathfrak{se}(2)$. In contrast to the exponential, the Cayley map has the advantage that it is an algebraic map, so that it is easily computable.

The Cayley map shares with the exponential map a number of useful properties, which will be used in some of the derivations below:
\begin{equation} \label{cay_properties}
	D \Cay(0) = \mathrm{id}. \quad \text{and} \quad \Cay(-\xi) = \Cay(\xi)^{-1},
\end{equation}
for all $\xi \in \mathfrak{se}(2)$.

\paragraph{The Right-Trivialized Derivative.}
For our purposes, we will need the \emph{right-trivialized derivative}
of the Cayley map, defined by
\begin{equation} \label{definition_cayley}
  \dCay_\xi(\eta) := (D \mathrm{Cay}(\xi)\cdot\eta) \mathrm{Cay}(\xi)^{-1};
\end{equation}
see \cite{IsMuNoZa2000}.  Note that $D \mathrm{Cay}(\xi)\cdot\eta$ is
an element of $T_\xi SE(2)$, which is translated back to
$\mathfrak{se}(2)$ by right-multiplication by
$\mathrm{Cay}(\xi)^{-1}$. In this way, $\mathrm{d} \mathrm{Cay}$ is a
map from $\mathfrak{se}(2) \times \mathfrak{se}(2)$ to
$\mathfrak{se}(2)$ which is linear in the second argument.

For fixed $\xi \in \mathfrak{se}(2)$, we denote the inverse of $\mathrm{d} \mathrm{Cay}_\xi$ by $\mathrm{d} \mathrm{Cay}_\xi^{-1}$. From \eqref{definition_cayley}, we have that
\begin{equation} \label{expl_dcay_inv}
	\mathrm{d} \mathrm{Cay}_\xi^{-1} (\eta) = [ D \Cay(\xi) ]^{-1} \cdot ( \eta \Cay(\xi) ).
\end{equation}

For the group $SE(2)$, the Cayley map and its derivatives were computed explicitly in
\cite{Ko2008}. Keeping in mind that the elements of $\mathfrak{se}(2)$
are represented as column vectors, for each $\xi \in
\mathfrak{se}(2)$, $\mathrm{d} \mathrm{Cay}_\xi^{-1}$ is a linear transformation from $\mathfrak{se}(2)$ to itself, given by 
\[
   \mathrm{d} \mathrm{Cay}_\xi^{-1}(\eta) = M(\xi) \eta
\]
where, for $\xi = (\omega, v, w)$, the matrix $M(\xi)$ is given by 
\begin{equation} \label{matrix_representation}
  M(\xi) = 		
    \begin{pmatrix}
      1 + \omega^2/4 & 0 & 0 \\
      -w/2 + \omega v/4 & 1 & \omega/2 \\
      v/2 + \omega w/4 & - \omega/2 & 1 
    \end{pmatrix}.
\end{equation}

For future reference, we record the following property of the right-trivialized derivative (see \cite{BoMa2009}): for all $\xi, \eta \in \mathfrak{se}(2)$
\begin{equation} \label{Ad_property}
	\mathrm{d} \mathrm{Cay}^{-1}_\xi(\eta) = 
		\mathrm{d} \mathrm{Cay}^{-1}_{-\xi} 
			\left( \mathrm{Ad}_{\mathrm{Cay}(-\xi)} \eta \right),
\end{equation}
where $\mathrm{Ad}$ is the adjoint action of $SE(2)$ on $\mathfrak{se}(2)$, defined by $\mathrm{Ad}_g(\xi) = g \xi g^{-1}$, where the elements on the right-hand side are interpreted as matrices, as in \eqref{matrix_group} and \eqref{matrix_algebra}.

Lastly, for each $\xi \in \mathfrak{se}(2)$, its adjoint,
$(\dCay^{-1}_\xi )^\ast$, is a linear map from
$\mathfrak{se}(2)^\ast$ to itself, defined by 
\begin{equation} \label{def_adjoint}
	\left< (\dCay^{-1}_\xi )^\ast \mu, \eta \right> 
		= 
	\left< \mu, \dCay^{-1}_\xi(\eta) \right>,
\end{equation}
relative to the duality pairing \eqref{pairing}. Explicitly,
\begin{equation} 
    (\dCay^{-1}_\xi )^\ast \mu
    = M(\xi)^T \mu.
\end{equation}

We will use the Cayley map to provide a parametrization of a neighborhood of the identity in $SE(2)$ by means of the Lie algebra $\mathfrak{se}(2)$, but it is possible to replace the Cayley map by any other local diffeomorphism satisfying \eqref{cay_properties} from $\mathfrak{se}(2)$ to $SE(2)$, such as the exponential map.  The Cayley map, however, has the advantage that it is efficiently computable, and its derivative is particularly easy to characterize.

\subsection{The Deformation Energy}

\paragraph{The Discrete Reconstruction Relations.}

Using the Cayley map, we discretize the reconstruction relations \eqref{reconstr_eq} as follows. Given two successive elements $g_k, g_{k+1}$ in $SE(2)$, we define the {\bfi update element} $\xi_k \in \mathfrak{se}(2)$ by 
\begin{equation} \label{update_def}
	\mathrm{Cay}(h \xi_k) = g_k^{-1} g_{k+1}.
\end{equation}
This is the discrete counterpart of the relation $\xi = g^{-1} g'$. Explicitly, if $\xi_k = (\omega_k, \mathbf{v}_k)$ and $g_i = (R_{\theta_i}, \mathbf{x}_i)$, $i = k, k+1$, we have for the components 
\begin{equation} \label{cayley_components}
	\hat{R}_{h\omega_k} = R_{\theta_{k+1} - \theta_k}, \quad \text{and} \quad
	\hat{\mathbf{x}}_{h\xi_k} = R_{-\theta_k}(\mathbf{x}_{k+1} - \mathbf{x}_k).
\end{equation}
The first relation is equivalent to the following trigonometric relation: 
\begin{equation} \label{cayley_trig}
	\frac{h \omega_k}{2} = \tan \left( \frac{\theta_{k+1} - \theta_k}{2} \right).
\end{equation}

\paragraph{The Deformation Energy.} 

To discretize the deformation energy, we now proceed as in the continuous case. We define $E$ as 
\begin{equation} \label{discrete_deformation_energy}
  E = h\sum_{k=0}^{N-1} 
  \left( \frac{1}{2} 
    \llb \xi_k, \xi_k \rrb 
    + \lb \mu_k, \frac{1}{h} \mathrm{Cay}^{-1}(g_k^{-1} g_{k+1}) - \xi_k \rb
  \right),
\end{equation}
which can be viewed as the discrete counterpart of \eqref{def_energy}. Here, $g_k \in SE(2)$, $\xi_k \in \mathfrak{se}(2)$ and $\mu_k \in \mathfrak{se}(2)^\ast$ are independent variables.  As mentioned at the beginning of this section, this energy function was originally introduced in \cite{BoMa2009}, and the derivations up to \eqref{variations_admissibility}, when we have to enforce the discrete admissibility constraint, will follow that paper.

By taking variations with respect to $\mu_k$, we recover the definition \eqref{update_def} of the update element $\xi_k$. By taking variations with respect to $\xi_k$, we arrive at the equation $\mu_k = \xi_k^\flat$, or explicitly
\begin{equation}  \label{legendre}
	\pi_k = m \omega_k, \quad \text{and} \quad \mathbf{p}_k = \mathbf{v}_k.
\end{equation}

Lastly, by taking variations with respect to the group element $g_k$ we obtain 
\[
	\delta E = \sum_{k = 0}^{N-1} \lb \mu_k, D \Cay^{-1}(g_k^{-1} g_{k+1}) \cdot
		( - g_k^{-1} \delta g_k g_k^{-1} g_{k+1} + g_k^{-1} \delta g_{k+1} ) \rb,
\]
where we have used the fact that $\delta g_k^{-1} = - g_k^{-1} \delta g_k g_k^{-1}$. We now introduce the quantity $\sigma_k := g_k^{-1} \delta g_k$ and focus first on the first derivative term, which we write as 
\[
	D \Cay^{-1}(g_k^{-1} g_{k+1}) \cdot ( \sigma_k g_k^{-1} g_{k+1} ) 
		 = [D \Cay (h \xi_k)]^{-1} \cdot ( \sigma_k \Cay(h \xi_k)) 
		 = \dCay^{-1}_{h\xi_k} (\sigma_k),
\]
where we have used the definition \eqref{expl_dcay_inv} of $\dCay^{-1}$, together with the expression \eqref{update_def} for the update element. For the second term in  $\delta E$, we proceed along similar lines:
\begin{align*}
	D \Cay^{-1}(g_k^{-1} g_{k+1}) \cdot (g_k^{-1} \delta g_{k+1}) & = 
		D \Cay^{-1}(g_k^{-1} g_{k+1}) \cdot (g_k^{-1}g_{k+1} \sigma_{k+1}) \\
		& = [D \Cay (h \xi_k)]^{-1} \cdot 
			( \Ad_{\Cay (h \xi_k)}(\sigma_{k+1})\Cay (h \xi_k)) \\
		& = \dCay^{-1}_{h\xi_k} ( \Ad_{\Cay (h \xi_k)}(\sigma_{k+1})) \\
		& = \dCay^{-1}_{- h\xi_k} (\sigma_{k+1})
\end{align*}
where we have used the property \eqref{Ad_property} in the last step.

Substituting both of these expressions for the derivatives back into the expression for $\delta E$, we arrive at 
\begin{align*}
	\delta E & =
		\sum_{k=0}^{N-1} \left\{-\lb \mu_k, \dCay^{-1}_{h\xi_k} (\sigma_k)\rb 
			+ \lb \mu_k, \dCay^{-1}_{- h\xi_k} (\sigma_{k+1}) \rb \right\}
		\\
	 & = \sum_{k=0}^{N-1} \left\{
		- \lb (\dCay^{-1}_{h\xi_k})^\ast \mu_k, 
				\sigma_k \rb
		+ \lb (\dCay^{-1}_{-h\xi_k})^\ast \mu_k, 
				\sigma_{k+1} \rb
	\right\},
\end{align*}
where we have introduced the adjoint of the linear map $\mathrm{d}\mathrm{Cay}^{-1}_{h\xi_k})$ using the definition \eqref{def_adjoint}. We now rearrange the terms in the sum to get
\begin{equation}
\begin{split}
	\delta E = & - \lb (\mathrm{d}\mathrm{Cay}^{-1}_{h\xi_0})^\ast \mu_0, 
				\sigma_0 \rb + 
			 \lb (\mathrm{d}\mathrm{Cay}^{-1}_{-h\xi_{N-1}})^\ast \mu_{N-1}, 
				\sigma_{N} \rb \\
		& +
		 \sum_{k=1}^{N-1}
		 \lb - (\mathrm{d}\mathrm{Cay}^{-1}_{h\xi_k})^\ast \mu_k
		     + (\mathrm{d}\mathrm{Cay}^{-1}_{-h\xi_{k-1}})^\ast \mu_{k-1},
			\sigma_k \rb.
\end{split}
\end{equation}

It remains for us to obtain an expression for the variations $\sigma_k = g_k^{-1} \delta g_k$. Since $E$ is varied over all discrete admissible curves, \eqref{discrete_admissibility} must hold, and by differentiating and multiplying by $R_{-\theta_k}$, we find
\begin{equation} \label{variations_admissibility}
	- \delta \theta_k J (\mathbf{c}_0)_k + \mathbf{w}_k = 0,
\end{equation}
where $\delta \theta_k$ and $\mathbf{w}_k$ are the components of $\sigma_k$. Note that if $\delta g_k = (\delta \theta_k, \delta \mathbf{x}_k)$, then $\mathbf{w}_k = R_{-\theta_k} \delta \mathbf{x}_k$.  In other words, $\sigma_k$ is an element of $\mathfrak{se}(2)_{(\mathbf{c}_0)_k}$.  Since $\sigma_k$ is otherwise arbitrary, we arrive at the following weak form of the discrete equations of motion:
\begin{equation} \label{disc_weak_eom}
\lb - (\mathrm{d}\mathrm{Cay}^{-1}_{h\xi_k})^\ast \mu_k
		     + (\mathrm{d}\mathrm{Cay}^{-1}_{-h\xi_{k-1}})^\ast \mu_{k-1},
			\sigma_k \rb = 0
\end{equation}
for all $\sigma_k \in \mathfrak{se}(2)_{(\mathbf{c}_0)_k}$, together with the weak boundary conditions
\begin{equation} \label{disc_weak_bc}
	\lb (\mathrm{d}\mathrm{Cay}^{-1}_{h\xi_0})^\ast \mu_0, 
				\sigma_0 \rb = 0
		\quad \text{and} \quad  
	\lb (\mathrm{d}\mathrm{Cay}^{-1}_{-h\xi_{N-1}})^\ast \mu_{N-1}, 
				\sigma_{N} \rb = 0
\end{equation}
for all $\sigma_0 \in \mathfrak{se}(2)_{(\mathbf{c}_0)_0}$ and $\sigma_N \in \mathfrak{se}(2)_{(\mathbf{c}_0)_N}$.  Another way to formulate the equations of motion is to observe that the left-most factor in each of these contractions must take values in $\mathfrak{se}(2)^\circ_{(\mathbf{c}_0)_k}$, defined in \eqref{isotropy_annihilator}.  In this way, we arrive at the following theorem.

\begin{theorem}
	A discrete admissible curve $g_k \in SE(2)$, $k = 0, \ldots, N$, is a critical point of the deformation energy $E$ if and only if \eqref{disc_weak_eom} holds, with boundary conditions \eqref{disc_weak_bc}. This is equivalent to 
\begin{equation} \label{disc_strong_eom}
	- (\mathrm{d}\mathrm{Cay}^{-1}_{h\xi_k})^\ast \mu_k
		     + (\mathrm{d}\mathrm{Cay}^{-1}_{-h\xi_{k-1}})^\ast \mu_{k-1} 
	\in \mathfrak{se}(2)^\circ_{(\mathbf{c}_0)_k}
\end{equation}
together with the boundary conditions 
\[
	 (\mathrm{d}\mathrm{Cay}^{-1}_{h\xi_0})^\ast \mu_0
			\in \mathfrak{se}(2)^\circ_{(\mathbf{c}_0)_0} 
		\quad \text{and} \quad  
	(\mathrm{d}\mathrm{Cay}^{-1}_{-h\xi_{N-1}})^\ast \mu_{N-1}
			\in \mathfrak{se}(2)^\circ_{(\mathbf{c}_0)_N}. 
\]
\end{theorem}

\paragraph{The Discrete Constraint.}


The equation \eqref{disc_strong_eom} is a single scalar equation, which in itself is insufficient to determine all three components of $\xi_k = (\omega_k, \mathbf{v}_k)$. We now show that the admissibility condition \eqref{discrete_admissibility} gives rise to two further equations, allowing all three components of $\xi_k$ to be determined.

By taking the admissibility constraint for $k$ and subtracting it from the constraint for $k + 1$, we obtain (after multiplying from the left by $R_{-\theta_k}$) that
\[
	R_{\theta_{k+1} - \theta_k} (\mathbf{c}_0)_{k+1} - (\mathbf{c}_0)_k + 
		R_{-\theta_k}(\mathbf{x}_{k+1} - \mathbf{x}_k) = 
	R_{-\theta_k} ( (\mathbf{c}_1)_{k+1} - (\mathbf{c}_1)_k).
\]
Using the relation \eqref{cayley_components} for the components of the Cayley map, this becomes
\begin{equation} \label{discrete_constraint}
	\hat{R}_{h\omega_k} (\mathbf{c}_0)_{k+1} - (\mathbf{c}_0)_k + 
		\hat{\mathbf{x}}_{h\xi_k} = R_{-\theta_k} ( (\mathbf{c}_1)_{k+1} - (\mathbf{c}_1)_k).
\end{equation}
Given $\omega_k$, the first component of $\xi_k$, this relation can be solved to find the corresponding linear velocity $\mathbf{v}_k$. In fact, since $\hat{\mathbf{x}}_{h \xi_k}$ depends linearly on $\mathbf{v}_k$, \eqref{discrete_constraint} is just a linear equation for $\mathbf{v}_k$.

\paragraph{Summary.}

To convince ourselves that the equations derived so far are sufficient to determine the discrete curve $g_k$, $k = 0, \ldots, N$ completely, we summarize the equations of motion. A practical way to implement these equations will be given below in Section~\ref{sec:implementation}.

Assume that two successive elements $g_{k-1}, g_k \in SE(2)$ are given, which satisfy the admissibility condition \eqref{discrete_admissibility}. The equations allow for $g_{k+1}$ to be found as follows:
\begin{enumerate}
	\item Using the Cayley transform \eqref{update_def}, find $\xi_{k-1} \in \mathfrak{se}(2)$.
	\item Solve the following two equations simultaneously for $\xi_k = (\omega_k, \mathbf{v}_k)$: \eqref{disc_strong_eom} and \eqref{discrete_constraint}.
	\item Given $\xi_k$ and $g_k$, determine $g_{k+1}$ from the update relation \eqref{update_def}.
\end{enumerate}

\subsection{Practical Implementation of the Discrete Equations of Motion} \label{sec:implementation}

Using the projector $\mathbb{P}$ defined in \eqref{annih_projection}, we first write the equations of motion as 
\begin{equation} \label{projected_eom}
\mathbb{P}_{(\mathbf{c}_0)_k} 
\left(
	- M(\hat{\xi}_k)^T \hat{\mu}_k + M(-\hat{\xi}_{k-1})^T \hat{\mu}_{k-1} 
\right) = 0 
 \end{equation}
and the boundary conditions as
\begin{equation} \label{projected_bc}
\mathbb{P}_{(\mathbf{c}_0)_0} 
\left(
M(\hat{\xi}_0)^T \hat{\mu}_0
\right) = 0, \quad \text{and} \quad
\mathbb{P}_{(\mathbf{c}_0)_N}
\left(
M(\hat{\xi}_{N-1})^T \hat{\mu}_{N-1}.
\right) = 0.
\end{equation}
Here we use the matrix expression $M(\xi)$ for $\mathrm{d} \mathrm{Cay}^{-1}$, given by \eqref{matrix_representation} and $\hat{\xi}_k = h \xi_k$, $\hat{\mu}_k = h \mu_k$.   From now on, we will drop the hat over the linear quantities $\xi$ and $\mu$, as the factors of $h$ in front of $\xi$ and $\mu$ can be restored at a later stage without any ambiguity.

\paragraph{The Discrete Equations of Motion for $SE(2)$.}

We introduce the matrices
\begin{equation} \label{convenient_matrices}
	A(\omega) := \frac{1}{1 + \omega^2/4} 
		\begin{pmatrix}
			1 & - \omega/2 \\
			\omega/2 & 1
		\end{pmatrix},
	\quad 
	B(\omega) := A(\omega)^{-1} =
		\begin{pmatrix}
			1 & \omega/2 \\
			-\omega/2 & 1
		\end{pmatrix},
\end{equation}
and observe that 
\[
	M(\xi)^T \mu = 
		\begin{pmatrix}
			\left(1 + \frac{\omega^2}{4} \right) \pi 
				- \frac{1}{2} \mathbf{p}^T B(\omega) J \mathbf{v} \\
			B(\omega)^T \mathbf{p} 
		\end{pmatrix}.
\]
Using the fact that $B(\omega)J = J - \frac{\omega}{2} I$, we obtain for the projection \eqref{annih_projection} that for any point $\mathbf{c}$
\begin{equation} \label{projected_expression}
	\mathbb{P}_{\mathbf{c}} (M(\xi)^T \mu) = 
		\left(1 + \frac{\omega^2}{4} \right) \pi 
		- \frac{1}{2} \mathbf{p}^T J \mathbf{v} 
		+ \frac{\omega}{4} \mathbf{p}^T \mathbf{v}
		+ \mathbf{p}^T J \mathbf{c} - \frac{\omega}{2} \mathbf{p}^T \mathbf{c}.
\end{equation}

Using \eqref{legendre} to write the momenta in terms of $\omega$ and $\mathbf{v}$, we obtain for the projected equation of motion \eqref{projected_eom} 
\begin{multline} \label{explicit_se_equation}
		m\left(1 + \frac{\omega_k^2}{4} \right) \omega_k 
		+ \frac{\omega_k}{4} \left\Vert \mathbf{v}_k \right\Vert^2
		+ \mathbf{v}_k^T J (\mathbf{c}_0)_k -
		\frac{\omega_k}{2} \mathbf{v}_k^T (\mathbf{c}_0)_k\\
		= 
		m\left(1 + \frac{\omega_{k-1}^2}{4} \right) \omega_{k-1} 
			+ \frac{\omega_{k-1} }{4} \left\Vert \mathbf{v}_{k-1} \right\Vert^2
		+ \mathbf{v}_{k-1}^T J (\mathbf{c}_0)_k + 
		\frac{\omega_{k-1}}{2} \mathbf{v}_{k-1}^T (\mathbf{c}_0)_k.
\end{multline}

\paragraph{The Linear Velocities.}

To solve the discrete constraint \eqref{discrete_constraint} for $\mathbf{v}$, we let 
\begin{equation} \label{quantity_b}
	\mathbf{b}_k := R_{-\theta_k} ( (\mathbf{c}_1)_{k+1} - (\mathbf{c}_1)_k)- \hat{R}_{\omega_k} (\mathbf{c}_0)_{k+1} + (\mathbf{c}_0)_k,
\end{equation}
so that \eqref{discrete_constraint} is equivalent to $A(\omega_k) \mathbf{v}_k = \mathbf{b}_k$, which can be solved as
\begin{equation} \label{velocity_relation}
	\mathbf{v}_k = B(\omega_k) \mathbf{b}_k.
\end{equation}
Notice that the right-hand side depends only on $\omega_k$ and $\theta_k$.

\paragraph{The Boundary Conditions.} 

Lastly, we show how the boundary conditions \eqref{projected_bc} can be made more explicit.  We assume that the two discrete curves have been translated to the origin, so that $(\mathbf{c}_0)_0 = (\mathbf{c}_1)_0 = 0$. Using \eqref{projected_expression}, the boundary condition for $k = 0$ becomes
\[
	\omega_0 \left( m + \frac{1}{4}(m \omega_0^2 + 
		\left\Vert \mathbf{v}_0 \right\Vert^2 )  \right) = 0,
\]
so that $\omega_0 = 0$.

To find $\mathbf{v}_0$, we focus on the discrete constraint \eqref{discrete_constraint}, which becomes for $k = 0$
\[
	\hat{R}_{\omega_0} (\mathbf{c}_0)_1  + 
		\hat{\mathbf{x}}_{\xi_0} = R_{-\theta_0} (\mathbf{c}_1)_1,
\]
so that at  $k = 0$, the following conditions hold:
\begin{equation} \label{leftmost_bc}
	\mathbf{x}_0 = 0, \quad \omega_0 = 0, 
		\quad \mathbf{v}_0 = R_{-\theta_0} (\mathbf{c}_1)_1 - (\mathbf{c}_0)_1,
\end{equation}
and where $\theta_0$ is arbitrary.  Once $\theta_0$ is chosen, these relations suffice to find the first two group elements $g_0$ and $g_1$.

At the other end of the curve, the boundary condition \eqref{projected_bc} for $k = N$ reads
\begin{equation} \label{terminal_bc}
m \left(1 + \frac{\omega_{N-1}^2}{4} \right) \omega_{N-1}
		+ \frac{\omega_{N-1}}{4} \left\Vert \mathbf{v}_{N-1} \right\Vert^2
		+ \mathbf{v}_{N-1}^T J (\mathbf{c}_0)_N - \frac{\omega_{N-1}}{2} \mathbf{v}_{N-1}^T (\mathbf{c}_0)_N = 0.
\end{equation}

\paragraph{Summary.} To summarize, the discrete equations of motion can be solved as follows. Given an element $(\theta_k, \mathbf{x}_k) \in SE(2)$, we may find the subsequent element $(\theta_{k+1}, \mathbf{x}_{k+1}) \in SE(2)$ by first solving the equations of motion \eqref{explicit_se_equation} for $\theta_{k+1}$, where $\omega_k$ has been eliminated using the Cayley relation \eqref{cayley_trig} and $\mathbf{v}_k$ using the linear relation \eqref{velocity_relation}. Afterwards, we then compute $\mathbf{x}_{k+1}$ from $\theta_{k+1}$ using the admissibility constraint.

Given an initial condition for $\theta_0$, the leftmost boundary conditions \eqref{leftmost_bc} can be used to find $\omega_0, \mathbf{x}_0$, and $\mathbf{v}_0$. We may then solve the discrete equations in motion to obtain $(\theta_k, \mathbf{x}_k)$ for $k = 1, \ldots, N$, until we arrive at $k = N$. Starting with arbitrary initial data for $\theta_0$, the corresponding solution will in general not satisfy the terminal boundary condition \eqref{terminal_bc}. Below, we outline a shooting algorithm which will allow us to adjust $\theta_0$ so as to satisfy the terminal boundary condition.

\subsection{Solving the Boundary Value Problem} \label{sec:BVP}

\paragraph{First-Variation Equations.}
We now solve the boundary value problem \eqref{projected_eom}, \eqref{projected_bc} using a simple Newton iteration. To this end, we begin by linearizing the equations \eqref{explicit_se_equation} around a given solution. We denote 
\[
	c_k(\omega, \mathbf{v}) := m \left( 1 + \frac{3}{4} \omega^2 \right)  + 
		\frac{\left\Vert \mathbf{v} \right\Vert^2}{4} - 
		\frac{1}{2} \mathbf{v}^T (\mathbf{c}_0)_k 
\]
and
\[
	\mathbf{d}_k(\omega) := \frac{\omega}{2} \mathbf{v} + J (\mathbf{c}_0)_k
	-  \frac{\omega}{2} (\mathbf{c}_0)_k.
\]
The first-variation equation may then be expressed as
\begin{equation} \label{first_var}
	c_k(\omega_k, \mathbf{v}_k) \delta \omega_k + 
	\mathbf{d}_k(\omega_k)^T \delta \mathbf{v}_k = 
	c_k(-\omega_{k-1}, -\mathbf{v}_{k-1}) \delta \omega_{k-1} + 
	\mathbf{d}_k(-\omega_{k-1})^T \delta \mathbf{v}_{k-1}.
\end{equation}
This is a single linear equation for $\delta \omega_k$; $\delta \mathbf{v}_k$ can be obtained from the linearization of \eqref{velocity_relation}. After some algebra, we obtain 
\[
	\delta \mathbf{v}_k = J\left( \frac{1}{2} \mathbf{b}_k + 
		A(\omega_k)  (\mathbf{c}_0)_{k+1}  \right) 
			\delta \omega_k + B(\omega_k) R_{-\theta_k} J 
			( (\mathbf{c}_1)_{k+1} - (\mathbf{c}_1)_k ) \delta \theta_k,
\]
where $A(\omega)$ was defined in \eqref{convenient_matrices}.
Lastly, the variation $\delta \theta_k$ may be obtained from the linearization of the Cayley equation \eqref{cayley_trig} and is given by
\[
	\delta \theta_{k+1} - \delta \theta_k = 
		\frac{1 - \omega_k^2/4}{1 + \omega_k^2/4} \delta \omega_k.
\]

\paragraph{Newton Iteration.} Starting with a value $\theta_{\mathrm{guess}}$, we set $\theta_0 = \theta_{\mathrm{guess}}$ we solve the equations of motion \eqref{explicit_se_equation} for $k = 0, \ldots, N-1$, and we compute the first variational quantities \eqref{first_var} along the trajectory. At the terminal end of the curve, we compute 
\[
	\delta \theta_{\mathrm{guess}} = 
	\frac{m \left(1 + \frac{\omega_{N-1}^2}{4} \right) \omega_{N-1}
		+ \frac{\omega_{N-1}}{4} \left\Vert \mathbf{v}_{N-1} \right\Vert^2
		+ \mathbf{v}_{N-1}^T J (\mathbf{c}_0)_N - \frac{\omega_{N-1}}{2} \mathbf{v}_{N-1}^T (\mathbf{c}_0)_N}{c_N(\omega_{N-1}, \mathbf{v}_{N-1}) \delta \omega_{N-1} + 
	\mathbf{d}_N(\omega_{N-1})^T \delta \mathbf{v}_{N-1}};
\]
that is, \eqref{terminal_bc} divided by its linearization, and we update $\theta_{\mathrm{guess}}$ by 
\[
	\theta_{\mathrm{guess}} \mapsto \theta_{\mathrm{guess}} - 
		\delta \theta_{\mathrm{guess}}
\]
to obtain our new starting value. The algorithm typically converges to a solution of the boundary value problem after only a few iterations.

\section{Direct Minimization of the Energy Functional}

Instead of explicitly solving the boundary value problem \eqref{projected_eom}, one can also  minimize the deformation energy \eqref{discrete_deformation_energy} directly.  To bring $E$ into a form which can be handled conveniently by standard optimization software, we write it as 
\begin{equation} \label{energy_angles}
	h E_0(\theta_0, \ldots, \theta_N) = \frac{1}{2} \sum_{k = 0}^{N-1} 
		\left( m \omega_k^2 + \left\Vert \mathbf{v}_k\right\Vert^2 \right).
\end{equation}
Here, we recall that the linear quantities $\omega_k$ and $\mathbf{v}_k$ have been scaled by $h$; this explains the factor $h$ in front of $E$ on the left-hand side.  Recall that $\omega_k$ is given in terms of the angles $\theta_l$ by \eqref{cayley_trig}, while $\mathbf{v}_k$ is given by \eqref{velocity_relation}. Equivalently, $E_0$ can be written as 
\[
h E_0(\theta_0, \ldots, \theta_N) = \frac{1}{2} \sum_{k = 0}^{N-1} 
		\left( m \omega_k^2 + \left(1 + \frac{\omega_k^2}{4} \right) \left\Vert \mathbf{b}_k\right\Vert^2 \right),	
\]
where $\mathbf{b}_k$ is given by \eqref{quantity_b}.

The gradient of $E_0$ with respect to $\theta_i$ can easily be computed from this expression. A standard calculation yields
\[
	\frac{\partial h E_0}{\partial \theta_k} = 
		Q_{k-1} - Q_{k}
		+ \left( 1 + \frac{\omega_k^2}{4} \right)
			\mathbf{f}_k \cdot \mathbf{g}_k
\]
for $k = 1, \ldots, N-1$, where 
\[
Q_i =  \left( 1 + \frac{\omega_i^2}{4} \right)
\left[
	\left( m  + \frac{\left\Vert \mathbf{b}_i \right\Vert^2}{4} \right)\omega_i 
	- \left( 1 + \frac{\omega_i^2}{4} \right) 
		\mathbf{b}_i \cdot \frac{d \hat{R}_\omega}{d \omega}(\omega_i) 
				(\mathbf{c}_0)_{i+1} \right]
\]
and
\[
\mathbf{f}_i  = - \hat{R}_{\omega_i} (\mathbf{c}_0)_{i+1} + (\mathbf{c}_0)_i,
	\quad \quad
\mathbf{g}_i = J R_{-\theta_i} ( (\mathbf{c}_1)_{i+1} - (\mathbf{c}_1)_i ).
\]

At the terminal points, we have
\[
\frac{\partial h E_0}{\partial \theta_0} = 
	 - Q_{0}
		+ \left( 1 + \frac{\omega_0^2}{4} \right)
			\mathbf{f}_0\cdot \mathbf{g}_0
	\quad \text{and} \quad
\frac{\partial h E_0}{\partial \theta_N} = 
		Q_{N-1}.
\]

\section{Numerical Results}\label{sec:num}

Unless otherwise noted, the parameter $m$ in the deformation energy \eqref{discrete_deformation_energy} will be taken to be equal to $2$ in the numerical experiments below.

\subsection{Discrepancy of Simple Planar Curves} 

\begin{figure}
\begin{center}
	\includegraphics{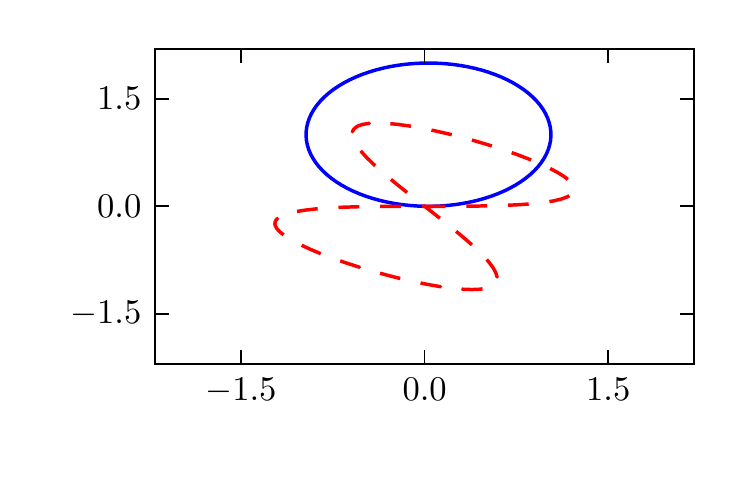}
	\includegraphics{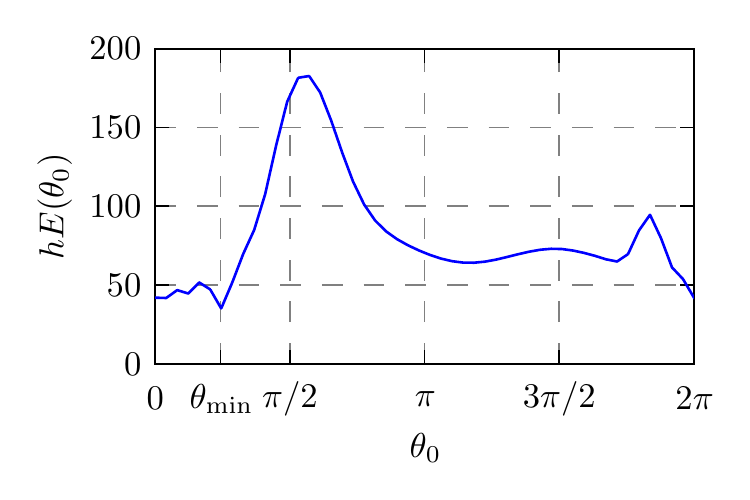}
	\caption{Left: we consider relative geodesics between the circle in blue (solid line) and the figure eight rotated over $\pi/3$ in red (dashed line). Right: the energy as a function of the initial condition $\theta_0$ has several minima. The global minimum is attained for $\theta_{\mathrm{min}} \approx 0.76$, where the energy is approximately $35.12$. \label{fig:circle_to_eight_theta_vs_energy}}
\end{center}
\end{figure}

\paragraph{Relative geodesics between the unit circle and a figure-eight shape.}

In this example, we compute the discrepancy between the unit circle with parametric representation $\mathbf{c}_0 = (\cos(2\pi s), \sin(2 \pi s))$ and the figure eight given by $\mathbf{c}_1(s) = (\sin(4\pi s), \sin (2\pi s))$, for $s \in [0, 1]$. We divide the parameter interval into  $N = 100$ equal subintervals, and we let $(\mathbf{c}_i)_k = \mathbf{c}_i(hk)$ for $i = 0, 1$ and $k = 0, \ldots, N$, where $h = 1/N$. We finish by rotating and translating both curves so that they start at the origin and are tangent to the $x$-axis. The final layout of the curves is shown in Figure~\ref{fig:circle_to_eight_theta_vs_energy}. Upon solving the boundary value problem, we find that the global minimum of the deformation energy is located at $\theta_{\mathrm{min}} \approx 0.7626$, where $\delta(\mathbf{c}_0, \mathbf{c}_1) = E(\theta_{\mathrm{min}}) \approx 35.1236$.

\begin{table}
\begin{center}
\begin{tabular}{c|c}
$N$ & $\delta(\mathbf{c}_0, \mathbf{c}_1)$ \\
\hline \hline
10  &  \underline{3}7.07641 \\
100  &  \underline{35.1}2402 \\
1000  &  \underline{35.146}75 \\
10000  &  35.14698 \\
\end{tabular}
\caption{Whenever the amount of interpolation points increases by a factor of 10, the accuracy of the discrepancy increases by two digits, suggesting that the order of accuracy of our numerical method is 1. \label{table:convergence}}
\end{center}
\end{table}

\paragraph{Convergence of the discrepancy.}

Last, we use the problem of finding the discrepancy between the circle and the figure eight to get a rough estimate of the order of accuracy of our numerical method in terms of the number of sample points on the discrete curves. To do this, we run a number of simulations: at each stage, we choose $N + 1$ sample points on each curve, and we compute the discrepancy between $\mathbf{c}_0$ and $\mathbf{c}_1$. As $N$ increases, the discrepancy is expected to approach a limit value, and the rate at which convergence takes place will give us an estimate of the order of accuracy of our numerical method.

In Table~\ref{table:convergence}, we have listed the discrepancy for a few choices of $N$. Roughly speaking, as $N$ goes up by a factor of 10, the discrepancy gains two digits of accuracy, so the order of accuracy of the method is approximately one.

\subsection{Interpolation for Curves on $SE(2)$} \label{sec:interpolation}

To provide some further visual insight into the nature of relative geodesics, we use a simple form of interpolation on $SE(2)$. Given an element $g \in SE(2)$, we define for $\epsilon \in [0, 1]$,
\[
	g_\epsilon = \exp(\epsilon \log(g)), 
\]
where $\exp : \mathfrak{se}(2) \to SE(2)$ is the exponential map and $\log = \exp^{-1}$ its inverse. The element $g_\epsilon$ will be well-defined provided that $g$ is in the range of the exponential map. In this way, we obtain a smooth curve in $SE(2)$ which connects the identity element to $g$ as we let $\epsilon$ range from 0 to 1.

Given a curve $s \mapsto g(s)$, we may apply this transformation to each point of the curve in order to obtain a family of curves $g_\epsilon(s)$. Now, assume that the original curve $g(s)$ matches the planar curves $\mathbf{c}_0$ and $\mathbf{c}_1$ and set $\mathbf{c}_\epsilon(s) = g_\epsilon(s) \cdot \mathbf{c}_0$.  As $\epsilon$ ranges from 0 to 1, $\mathbf{c}_\epsilon$ will be a curve which smoothly interpolates between $\mathbf{c}_0$ and $\mathbf{c}_1$.  A similar procedure may be done for the case of discrete curves.

Most of the computations in the remainder of this section were visualized using the following form of linear interpolation: whenever we compute an admissible curve $g(s)$ with respect to two planar curves $\mathbf{c}_0$ and $\mathbf{c}_1$, we construct the intermediate curves $\mathbf{c}_\epsilon$ to give an idea of the deformations carried out by $g(s)$. While the sequence of curves thus obtained has no immediate physical meaning, it nevertheless gives a good intuitive idea of the amount of deformation needed to match one curve with another. To illustrate this, we show in Figure~\ref{fig:matching-comparison} the matching between a circle and a figure-eight shape, where the matching is first done using the global minimizer of the energy~\eqref{def_energy_group}, and secondly with two different local minimizers. 

\begin{figure}
\begin{center}
	\subfloat[][Global minimum, $\theta_0 = 0.7626$, $E(\theta_0) = 35.1236$.]%
		{\includegraphics[scale=.4]{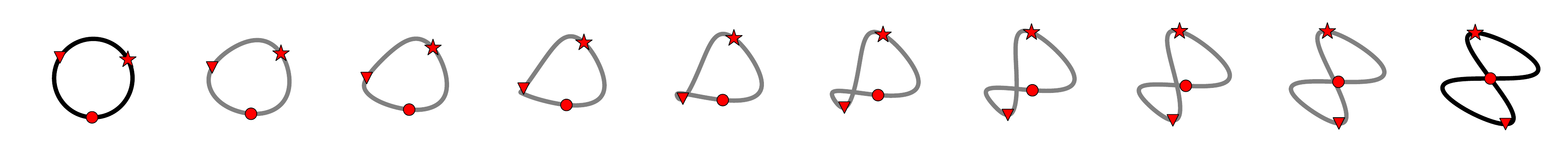}} \\
	\subfloat[][Local minimum, $\theta_0 = 0.3608$, $E(\theta_0) = 44.4211$.]%
		{\includegraphics[scale=.4]{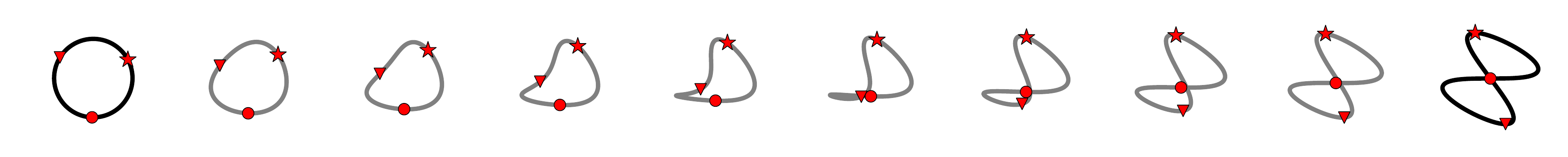}} \\
	\subfloat[][Local minimum, $\theta_0 = 5.3777$, $E(\theta_0) = 64.9445$.]%
		{\includegraphics[scale=.4]{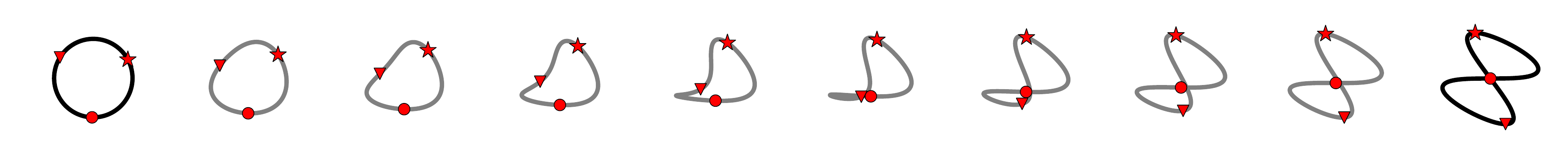}}
	\caption{Matching of a circle and a figure-eight shape using the global minimizer of the deformation energy (Figure~(a)) and using two local minimizers (Figures~(b) and~(c)). The global minimizer deforms the circle into the figure eight with considerably less deformation than the two local minimizers. To indicate the motion of the individual points on the curve, the marks  indicate respectively the first point on the curve (circle), the point at index $k = N/3$ (asterisk), and the point at index $k = 2N/3$ (triangle). \label{fig:matching-comparison}}
\end{center}
\end{figure}

\subsection{Asymmetry of the Discrepancy}

We have mentioned before that the discrepancy is not symmetric. We illustrate this by matching a circle $\mathbf{c}_0$ of radius $r = 0.1$ with a figure-eight shape, given by the parametric representation $\mathbf{c}_1(s) = (\sin(4\pi s), \sin(2 \pi s))$. On both curves, 100 points were sampled equidistantly in $s$, and the parameter $m$ in the norm \eqref{def_energy_group} was set to 2. We find the optimal admissible curve $g$ by means of the algorithm in Section~\ref{sec:BVP}. For the discrepancy, we obtain 
\[
	\delta(\mathbf{c}_0, \mathbf{c}_1) = 47.6261
	\quad \text{and} \quad
	\delta(\mathbf{c}_1, \mathbf{c}_0) = 39.8011.
\]

The deformations leading to these respective discrepancies are visualized in Figure~\ref{fig:asym-visualizations}, using the interpolation procedure described in Section~\ref{sec:interpolation}.

\begin{figure}
\begin{center}
	\includegraphics[scale=.4]{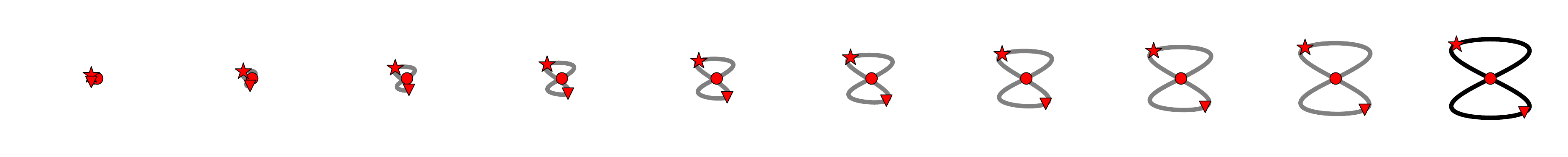} \\
	\includegraphics[scale=.4]{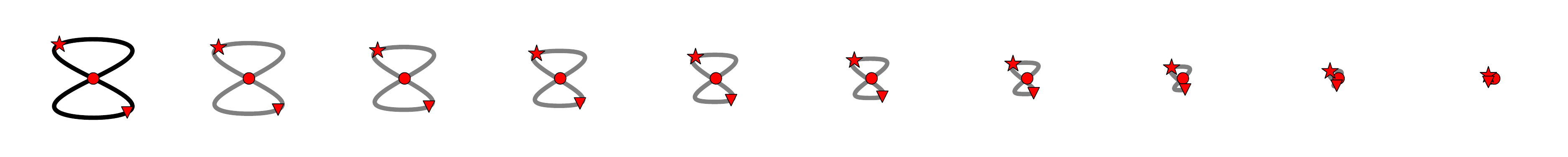}
	\caption{Interpolation between circle and figure-eight (top) and figure-eight and circle (bottom). \label{fig:asym-visualizations}}
\end{center}
\end{figure}

As these deformations appear quite similar, despite the large difference in the discrepancies, we investigate some further characteristics of the minimizer. In Figure~\ref{fig:theta-coordinate} we plot the $\theta$-coordinate of the minimizer for both matching problems. These curves appear quite different.

\begin{figure}
\begin{center}
	\includegraphics{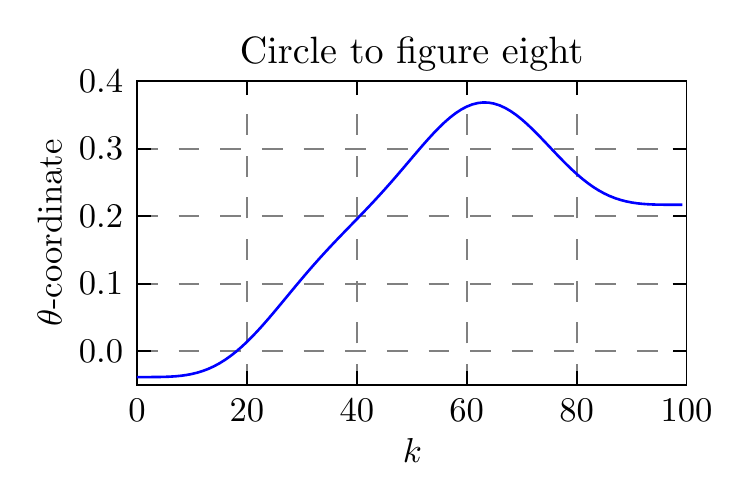}
	\includegraphics{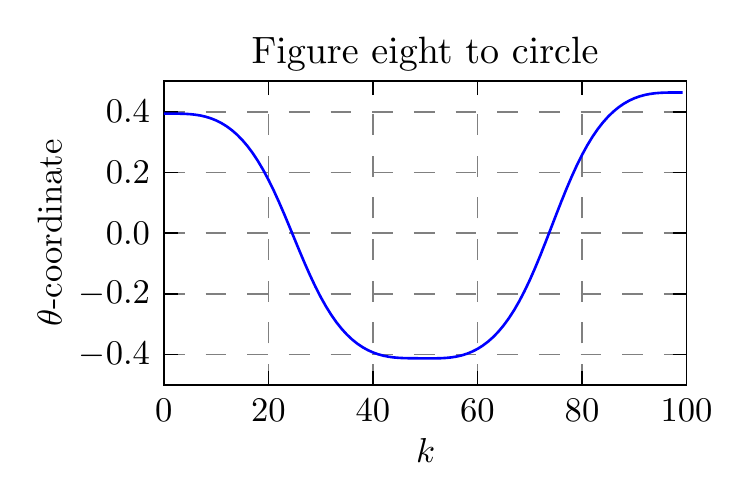}
	\caption{Theta-coordinate for the discrete curve matching the dot to the figure-eight (left), and for the figure-eight matching to the dot (right). These curves are quite different, indicating that the points on the dot, resp. the figure eight, undergo different rotations. \label{fig:theta-coordinate}}
\end{center}
\end{figure}

\subsection{Discrepancy of Polynomial Curves}

In this last example, we take a closer look at the relation between discrepancy and other geometrical invariants, in particular the total absolute curvature $\kappa$, defined as the integral of the absolute value of the curvature:
\[
	\kappa = \int_{s_0}^{s_1} \left| \rho(s) \right| \, ds.
\]

\begin{figure}
\begin{center}
	\includegraphics{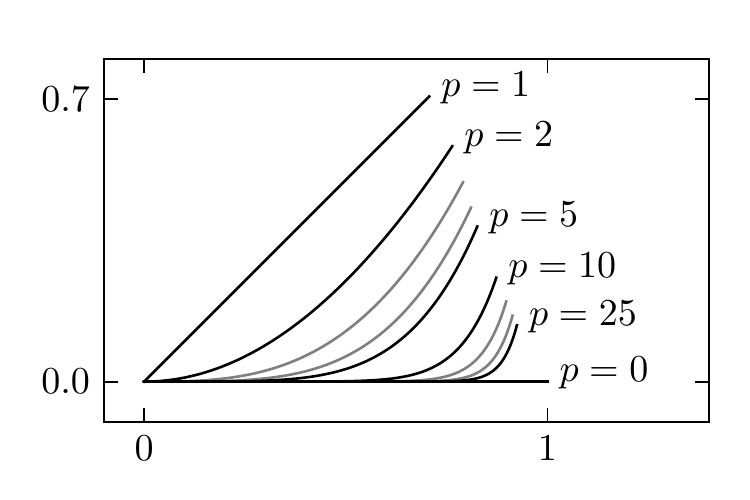}
	\includegraphics{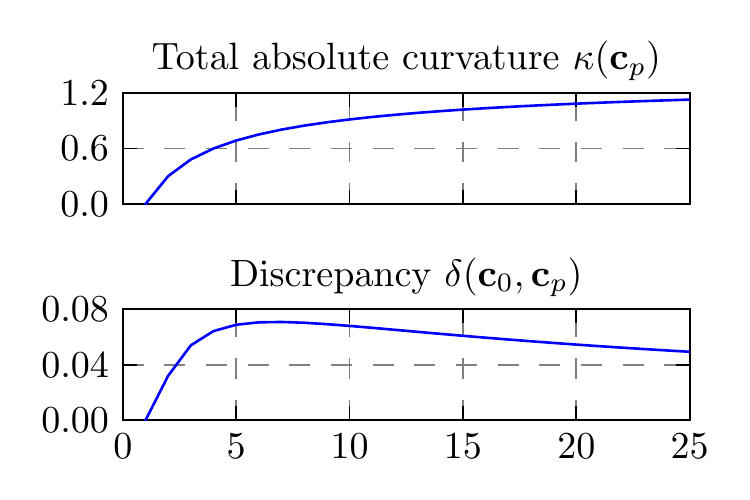}
	\caption{Left: Polynomial curves $\mathbf{c}_p : y = x^p$ for various exponents $p$.  Right: Total absolute curvature of $\mathbf{c}_p$ (top) and discrepancy between $\mathbf{c}_0$ and $\mathbf{c}_p$ (bottom) as a function of the exponent $p$. Whereas the total absolute curvature is an increasing function of $p$, the discrepancy reaches a maximum around $p = 6$ and then decreases.
	\label{fig:polycurves}}
\end{center}
\end{figure}

We focus on polynomial curves $y = x^p$, for $p = 1, 2, \ldots$ and set $\mathbf{c}_p(s) = (x_p(s), x_p(s)^p)$, where the parameter $s$ denotes arclength, $s = 0$ corresponds to the origin, and the curve is traversed in the direction of the positive $x$-axis (i.e. $x_p'(s) > 0$). Note that the $x_p(s)$-component is implicitly determined from the equation 
\[
	x'(s) = \frac{1}{\sqrt{1 + p^2 x(s)^{2p-2}}},
\]
which can easily be solved for various $p$ by numerical quadrature. We will focus on the unit length segment corresponding to $s \in [0, 1]$. A few of these curves are plotted in Figure~\ref{fig:polycurves}.

For each of these curves $\mathbf{c}_p(s)$, we choose $N$ equidistant points $s_k$ in the parameter interval $[0, 1]$, so that $s_k = (k-1)/(N-1)$ for $k = 1, \ldots, N$, and we set $(\mathbf{c}_p)_k = \mathbf{c}_p(s_k)$. In this way, we obtain for each exponent $p$ a discrete curve consisting of $N$ points $(\mathbf{c}_p)_k$, for $k = 1, \ldots, N$, at equal distance (along the curve) from each other. 

For each exponent $p > 0$, we compute the discrepancy between the discrete curve $(\mathbf{c}_p)_k$ and the fixed curve $(\mathbf{c}_0)_k$, which is parallel to the $x$-axis. We have tabulated the results for a selection of exponents $p$ in Table~\ref{table:discrepancy_vs_curvature}, along with the total geodesic curvatures for each of the curves $\mathbf{c}_p$. In Figure~\ref{fig:polycurves}, both invariants have been plotted for $p = 1, 2, \ldots, 25$. We see that the geodesic curvature increases as a function of $p$, corresponding to the more pronounced bend in the curve for higher $p$. On the other hand, the discrepancy first increases until $p = 6$ and then decreases: while the curves for high $p$ are more curved, the curvature is more localized and the curves as a whole are close to the $x$-axis, resulting in lower discrepancy.

\begin{table}
\begin{center}
\begin{tabular}{c|c|c|c|c|c|c|c|c|c}
$p$ & 1 & 2 & 3 & 4 & 5 & 10 & 15 & 20 & 25 \\
\hline \hline
$\delta(\mathbf{c}_0, \mathbf{c}_p)$ & 
	0 & 0.032 & 0.054 & 0.064 & 0.069 &
 	0.068 & 0.061 & 0.055 & 0.049 \\
\hline
$\kappa(\mathbf{c}_p)$ & 
	0 & 0.301 & 0.481 & 0.600 & 0.685 &
	0.913 & 1.019 & 1.083 & 1.127
\end{tabular}
\caption{Table of the discrepancy $\delta$ and the total absolute curvature $\kappa$ for various exponents $p$. As $p$ increases, the geodesic curvature increases, whereas the discrepancy reaches a maximum value around $p = 5$ and then decreases.  \label{table:discrepancy_vs_curvature}}
\end{center}
\end{table}

\section{Conclusions and Outlook}

In this paper, we have outlined a new measure for the discrepancy between planar curves, based on deforming one curve into the other by means of parameter-dependent transformations with values in the Lie group $SE(2)$. We defined a relative geodesic in $SE(2)$ to be a curve of transformations which extremizes a certain energy functional while mapping the first curve into the second, and we defined the discrepancy to be the value of the energy associated to the minimizing relative geodesic.

One of the advantages of our approach is that it can be generalized in a straightforward manner to deal with, for instance, discrepancies and relative geodesics between other types of geometric objects, such as curves in 3D or two-dimensional images. Another direction for future research addresses the choice of Lie group of transformations. In this paper we considered the group $SE(2)$ of rotations and translations, but other groups acting on the plane can be treated similarly. For instance, one could imagine acting on the curves by means of shearing transformations and translations, and in this case the relevant group is the semi-direct product $SL(2) \circledS \mathbb{R}^2$, where $SL(2)$ is the group of $2 \times 2$ matrices with unit determinant. 

\paragraph{Acknowledgements.}

We would like to thank Jaap Eldering, Henry Jacobs, and David Meier for stimulating discussions and helpful remarks.

DH and JV gratefully acknowledge partial support by the European Research Council Advanced Grant 267382 FCCA. LN is grateful for support from this grant during a visit to Imperial College. JV is also grateful for partial support by the {\sc irses} project {\sc
geomech} (nr.\ 246981) within the 7th European Community Framework Programme,
and is on leave from a Postdoctoral Fellowship of the Research Foundation--Flanders (FWO-Vlaanderen).

\bibliography{rel-geodesics}
\bibliographystyle{abbrv}

\end{document}